\documentclass{amsart}

\usepackage{microtype}
\usepackage[margin=1in]{geometry}

\usepackage{amsmath}
\usepackage{amssymb}
\usepackage{amsthm}

\usepackage{dna-diagrams}
\usepackage{float}
\usepackage[colorlinks=true,linkcolor=blue,citecolor=blue,urlcolor=blue,pdfencoding=auto,psdextra]{hyperref}

\hypersetup{
  bookmarksnumbered,
  pdftitle  = {A Rigid Category of DNA Secondary Structures},
  pdfauthor = {Andr\'es Ortiz-Mu\~noz},
}

\newtheorem{theorem}{Theorem}[section]
\newtheorem{proposition}[theorem]{Proposition}

\theoremstyle{definition}
\newtheorem{definition}[theorem]{Definition}
\newtheorem{example}[theorem]{Example}
\theoremstyle{remark}
\newtheorem{remark}[theorem]{Remark}

\newcommand{\Sig}{\Sigma}
\newcommand{\eps}{\varepsilon}
\newcommand{\comp}[1]{\overline{#1}}
\newcommand{\dv}{{}^{\vee}}
\newcommand{\Hom}{\mathrm{Hom}}
\newcommand{\id}{\mathrm{id}}
\newcommand{\ev}{\mathrm{ev}}
\newcommand{\coev}{\mathrm{coev}}
\newcommand{\Ddna}{\mathcal{D}_{\mathrm{DNA}}}

\title{A Rigid Category of DNA Secondary Structures}
\author{Andr\'es Ortiz-Mu\~noz}
\thanks{Harvard Medical School, Boston, MA, USA. Email: \href{mailto:andortiz19@gmail.com}{andortiz19@gmail.com}.}
\date{}

\begin{document}

\begin{abstract}
We construct a strict pivotal monoidal category $\Ddna$ whose objects are DNA sequences (words over $\{A,C,G,T\}$) and whose morphisms are isotopy classes of typed noncrossing planar matchings---composed of through-strands and Watson--Crick-typed arcs---in a rectangle with source and target boundaries. The dual of a sequence is its reverse complement, evaluation and coevaluation are canonical duplex pairings, and the snake identities hold by planar isotopy. A bending correspondence identifies each morphism $x \to y$ with a secondary structure on the combined word $x\dv y$; in particular, the generalized elements $\eps \to w$ are exactly the non-pseudoknotted secondary structures on~$w$. Composition, viewed in this straightened picture, is computed by a \emph{zip-and-transfer} operation on complementary interfaces---a combinatorial rearrangement of base-pair connectivity of which toehold-mediated strand displacement is a kinetically specific instance. Because $\Ddna$ is rigid monoidal, it shares the categorical backbone of pregroup grammars and the DisCoCat framework for compositional semantics: a strong monoidal functor from a grammatical category to $\Ddna$ maps grammatical reductions to Watson--Crick base pairing and sentence meanings to secondary structures. We describe this functor and discuss connections to algorithmic self-assembly, composable strand-displacement circuits, and constructive dynamical systems.
\end{abstract}

\maketitle

\section{Introduction}

DNA nanotechnology and molecular computation rest on the predictable interactions governed by Watson--Crick base pairing. In this paper we construct a category $\Ddna$ (for DNA diagrams) whose objects are words over the nucleotide alphabet $\{A,C,G,T\}$ and whose morphisms are isotopy classes of typed noncrossing arc diagrams (non-pseudoknotted secondary structures). The dual of a sequence is its reverse complement, evaluation and coevaluation are canonical duplex pairings, and the snake identities hold by planar isotopy, making $\Ddna$ a strict pivotal monoidal category. Composition, in the straightened picture, is computed by a \emph{zip-and-transfer} operation on complementary interfaces---a general mechanism of which toehold-mediated strand displacement~\cite{Yurke2000, Zhang2009, Qian2011} is a kinetically specific instance.

The categorical structure underlying $\Ddna$---a rigid monoidal category---is the same structure that Lambek~\cite{Lambek1999} identified in natural language syntax through pregroup grammars, and that Coecke, Sadrzadeh, and Clark~\cite{Coecke2010} used to build compositional distributional semantics (DisCoCat). This structural coincidence yields a monoidal functor from grammar to DNA in which grammatical reduction maps to Watson--Crick base pairing; we treat this as one application of the rigid structure rather than its sole motivation.

Section~\ref{sec:dna} recalls the biology; Section~\ref{sec:rigid} reviews rigid monoidal categories; Section~\ref{sec:category} constructs $\Ddna$ and its pivotal structure; Section~\ref{sec:grammar} describes the functor from pregroup grammars to $\Ddna$; and Section~\ref{sec:remarks} discusses variants, connections, and open questions.

\section{DNA secondary structures}\label{sec:dna}

Deoxyribonucleic acid (DNA) is a polymer whose monomers, called nucleotides or bases, are drawn from the four-letter alphabet
\[
\Sig = \{A, C, G, T\}.
\]
A single-stranded DNA molecule is a finite sequence of bases, read by convention from the $5'$ to the $3'$ end. We write such a sequence as a word $w = b_1 b_2 \cdots b_n$ in the free monoid $\Sig^*$, with the empty word $\eps$ representing the absence of any bases.

The defining chemical feature of DNA is Watson--Crick complementarity: adenine ($A$) pairs with thymine ($T$), and cytosine ($C$) pairs with guanine ($G$). We encode this as an involution on the alphabet,
\[
\comp{A} = T, \quad \comp{T} = A, \quad \comp{C} = G, \quad \comp{G} = C,
\]
which extends to words by reversal and letterwise complementation. The \emph{reverse complement} of $w = b_1 b_2 \cdots b_n$ is
\[
w\dv = \comp{b_n}\, \comp{b_{n-1}} \cdots \comp{b_1}.
\]
Reversal reflects the antiparallel orientation of the two strands in a DNA duplex: a strand running $5' \to 3'$ pairs with its complement running $3' \to 5'$. When $w$ and $w\dv$ are brought together, each base in $w$ can pair with the corresponding base in $w\dv$, forming a double-helical duplex.

A \emph{secondary structure} on a sequence $w = b_1 \cdots b_n$ is a set of base pairs---pairs of positions $(i,j)$ with $i < j$ such that $b_i$ and $b_j$ are Watson--Crick complements---subject to two constraints:
\begin{enumerate}
  \item \textbf{Uniqueness.} Each position participates in at most one base pair.
  \item \textbf{No crossings.} If $(i,j)$ and $(k,l)$ are both base pairs with $i < k$, then either $j < k$ (the pairs are disjoint) or $i < k < l < j$ (one pair is nested inside the other). The forbidden pattern is $i < k < j < l$, which would represent a \emph{pseudoknot}\footnote{The non-pseudoknotted (planar) model is a widely used first approximation for DNA secondary structure and is often sufficient for stochastic modeling of combinatorial DNA dynamics over folded-state energy landscapes. For sufficiently long sequences, however, pseudoknots do occur; categorically, this suggests braided (rather than purely planar pivotal) extensions as a more general framework. Despite this limitation, non-pseudoknotted models have shown strong predictive value in many experimental regimes~\cite{Waterman1978,Yurke2000,Zhang2009}. In this paper we focus on developing the underlying pivotal structure under the standard non-pseudoknotted assumption.}.
\end{enumerate}

The standard representation is an \emph{arc diagram}: the bases $b_1, \ldots, b_n$ are arranged along a horizontal line, and each base pair $(i,j)$ is drawn as an arc in the upper half-plane connecting positions $i$ and $j$. The no-crossing condition means that arcs do not intersect.
The diagram below shows a simple hairpin with five nested base pairs closing a three-base loop, shown in both arc and folded views.

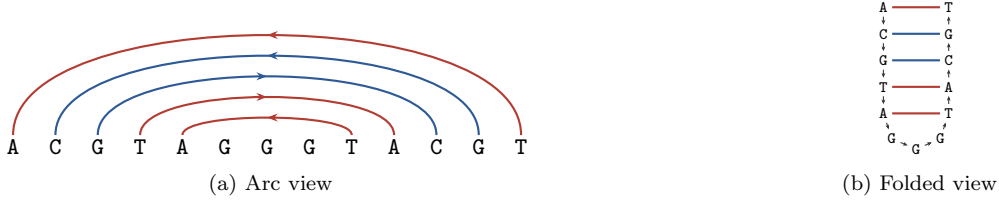
\begin{center}
\begin{minipage}[t]{0.48\textwidth}
\centering
\begin{tikzpicture}[x=0.56cm, y=0.56cm]
  \node[base node] (t0) at (1.00,0.00) {A};
  \node[base node] (t1) at (2.00,0.00) {C};
  \node[base node] (t2) at (3.00,0.00) {G};
  \node[base node] (t3) at (4.00,0.00) {T};
  \node[base node] (t4) at (5.00,0.00) {A};
  \node[base node] (t5) at (6.00,0.00) {G};
  \node[base node] (t6) at (7.00,0.00) {G};
  \node[base node] (t7) at (8.00,0.00) {G};
  \node[base node] (t8) at (9.00,0.00) {T};
  \node[base node] (t9) at (10.00,0.00) {A};
  \node[base node] (t10) at (11.00,0.00) {C};
  \node[base node] (t11) at (12.00,0.00) {G};
  \node[base node] (t12) at (13.00,0.00) {T};
  \draw[at arc rev] (t0.north) .. controls +(0,3.15) and +(0,3.15) .. (t12.north);
  \draw[cg arc rev] (t1.north) .. controls +(0,2.50) and +(0,2.50) .. (t11.north);
  \draw[cg arc] (t2.north) .. controls +(0,1.85) and +(0,1.85) .. (t10.north);
  \draw[at arc] (t3.north) .. controls +(0,1.20) and +(0,1.20) .. (t9.north);
  \draw[at arc rev] (t4.north) .. controls +(0,0.55) and +(0,0.55) .. (t8.north);
\end{tikzpicture}

\smallskip
\footnotesize (a) Arc view
\end{minipage}\hfill
\begin{minipage}[t]{0.48\textwidth}
\centering
\begin{tikzpicture}[x=0.54cm, y=0.54cm]
  \node[base node, font=\tiny\ttfamily] (b0) at (-0.800,-0.000) {A};
  \node[base node, font=\tiny\ttfamily] (b1) at (-0.800,-0.650) {C};
  \node[base node, font=\tiny\ttfamily] (b2) at (-0.800,-1.300) {G};
  \node[base node, font=\tiny\ttfamily] (b3) at (-0.800,-1.950) {T};
  \node[base node, font=\tiny\ttfamily] (b4) at (-0.800,-2.600) {A};
  \node[base node, font=\tiny\ttfamily] (b5) at (-0.594,-3.217) {G};
  \node[base node, font=\tiny\ttfamily] (b6) at (0.000,-3.480) {G};
  \node[base node, font=\tiny\ttfamily] (b7) at (0.594,-3.217) {G};
  \node[base node, font=\tiny\ttfamily] (b8) at (0.800,-2.600) {T};
  \node[base node, font=\tiny\ttfamily] (b9) at (0.800,-1.950) {A};
  \node[base node, font=\tiny\ttfamily] (b10) at (0.800,-1.300) {C};
  \node[base node, font=\tiny\ttfamily] (b11) at (0.800,-0.650) {G};
  \node[base node, font=\tiny\ttfamily] (b12) at (0.800,0.000) {T};
  \draw[-{Stealth[length=1.6pt,width=1.2pt]}, dnaWire, thin] (b0) -- (b1);
  \draw[-{Stealth[length=1.6pt,width=1.2pt]}, dnaWire, thin] (b1) -- (b2);
  \draw[-{Stealth[length=1.6pt,width=1.2pt]}, dnaWire, thin] (b2) -- (b3);
  \draw[-{Stealth[length=1.6pt,width=1.2pt]}, dnaWire, thin] (b3) -- (b4);
  \draw[-{Stealth[length=1.6pt,width=1.2pt]}, dnaWire, thin] (b4) -- (b5);
  \draw[-{Stealth[length=1.6pt,width=1.2pt]}, dnaWire, thin] (b5) -- (b6);
  \draw[-{Stealth[length=1.6pt,width=1.2pt]}, dnaWire, thin] (b6) -- (b7);
  \draw[-{Stealth[length=1.6pt,width=1.2pt]}, dnaWire, thin] (b7) -- (b8);
  \draw[-{Stealth[length=1.6pt,width=1.2pt]}, dnaWire, thin] (b8) -- (b9);
  \draw[-{Stealth[length=1.6pt,width=1.2pt]}, dnaWire, thin] (b9) -- (b10);
  \draw[-{Stealth[length=1.6pt,width=1.2pt]}, dnaWire, thin] (b10) -- (b11);
  \draw[-{Stealth[length=1.6pt,width=1.2pt]}, dnaWire, thin] (b11) -- (b12);
  \draw[dnaAT, thick] (b0) -- (b12);
  \draw[dnaCG, thick] (b1) -- (b11);
  \draw[dnaCG, thick] (b2) -- (b10);
  \draw[dnaAT, thick] (b3) -- (b9);
  \draw[dnaAT, thick] (b4) -- (b8);
\end{tikzpicture}

\smallskip
\footnotesize (b) Folded view
\end{minipage}
\end{center}

These structures are fundamental objects in molecular biology. They determine the geometry and function of single-stranded nucleic acids and have been studied combinatorially since the work of Waterman~\cite{Waterman1978}. Their enumeration is closely related to Motzkin numbers and noncrossing partitions. We will see that they also have a natural categorical interpretation.


\section{Rigid monoidal categories}\label{sec:rigid}

We recall the categorical notions used in the DNA construction. Readers familiar with rigid categories may skip to Section~\ref{sec:category}; we include definitions to fix notation and follow Selinger's conventions~\cite{Selinger2011}.

A \emph{monoidal category} $(\mathcal{C}, \otimes, I)$ consists of a category $\mathcal{C}$ with a bifunctor $\otimes \colon \mathcal{C} \times \mathcal{C} \to \mathcal{C}$ (tensor product), a unit object $I$, and associativity/unitality isomorphisms satisfying coherence conditions. It is \emph{strict} when these isomorphisms are identities; by Mac Lane coherence, every monoidal category is monoidally equivalent to a strict one, and we work strictly throughout.

\begin{definition}\label{def:rigid}
A strict monoidal category $(\mathcal{C}, \otimes, I)$ is \emph{right rigid} if every object $A$ is equipped with a \emph{right dual} $A^*$, together with morphisms
\[
\ev_A \colon A \otimes A^* \to I, \qquad \coev_A \colon I \to A^* \otimes A,
\]
called \emph{evaluation} (or \emph{cup}) and \emph{coevaluation} (or \emph{cap})\footnote{Different texts adopt different orientation and reading conventions for string diagrams (for example, bottom-to-top or left-to-right). Accordingly, the labels ``cup'' and ``cap'' are convention-dependent and may be interchanged in other references.}. We read string diagrams from top to bottom throughout this paper. With this convention, these maps satisfy the \emph{snake identities}:
\[
(\ev_A \otimes \id_A)\circ(\id_A \otimes \coev_A) = \id_A,
\]
\[
(\id_{A^*} \otimes \ev_A)\circ(\coev_A \otimes \id_{A^*}) = \id_{A^*}.
\]
It is \emph{left rigid} if analogous data exist on the other side, and \emph{rigid} (or \emph{autonomous}) if it is both left and right rigid.
\end{definition}

The snake identities have a vivid graphical interpretation in the language of string diagrams~\cite{Joyal1991}. Objects are drawn as labeled wires, morphisms as boxes or nodes, and the tensor product as horizontal juxtaposition. The unit object $I$ is represented by the empty diagram. In our pictures, the arrow direction on each wire encodes the duality structure: reversing orientation corresponds to passing between an object and its dual. In this notation, evaluation is a cup (a wire bending up from $A$ and $A^*$ to nothing) and coevaluation is a cap (a wire bending down from nothing to $A^*$ and $A$). The snake identities then assert that a zig-zag in a wire can be straightened:

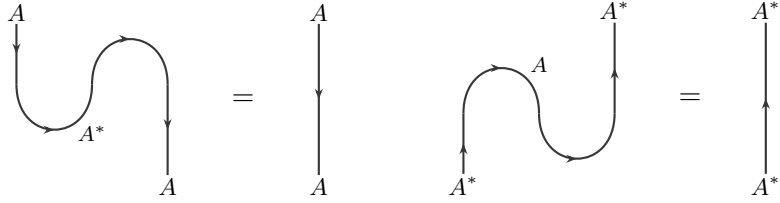
\begin{center}
\begin{tikzpicture}[x=1cm, y=1cm]
  \draw[rigid wire plain] (0,2) -- (0,1.2);
  \draw[rigid cup] (0,1.2) .. controls +(0,-0.8) and +(0,-0.8) .. (1,1.2);
  \draw[rigid cap] (1,1.2) .. controls +(0,0.8) and +(0,0.8) .. (2,1.2);
  \draw[rigid wire plain] (2,1.2) -- (2,0);
  \node[rigid node, above] at (0,2) {$A$};
  \node[rigid node, below] at (2,0) {$A$};
  \node[rigid node] at (1,0.55) {\footnotesize$A^*$};
  \node[rigid eq] at (3,1) {$=$};
  \draw[rigid wire] (4,2) -- (4,0);
  \node[rigid node, above] at (4,2) {$A$};
  \node[rigid node, below] at (4,0) {$A$};
\end{tikzpicture}
\qquad\qquad
\begin{tikzpicture}[x=1cm, y=1cm]
  \draw[rigid wire plain] (0,0) -- (0,0.8);
  \draw[rigid cap] (0,0.8) .. controls +(0,0.8) and +(0,0.8) .. (1,0.8);
  \draw[rigid cup] (1,0.8) .. controls +(0,-0.8) and +(0,-0.8) .. (2,0.8);
  \draw[rigid wire plain] (2,0.8) -- (2,2);
  \node[rigid node, below] at (0,0) {$A^*$};
  \node[rigid node, above] at (2,2) {$A^*$};
  \node[rigid node] at (1,1.45) {\footnotesize$A$};
  \node[rigid eq] at (3,1) {$=$};
  \draw[rigid wire] (4,0) -- (4,2);
  \node[rigid node, above] at (4,2) {$A^*$};
  \node[rigid node, below] at (4,0) {$A^*$};
\end{tikzpicture}
\end{center}

\begin{definition}
A rigid monoidal category is \emph{pivotal} if it is equipped with a monoidal natural isomorphism $A \xrightarrow{\sim} A^{**}$ for every object $A$. It is \emph{strictly pivotal} if this isomorphism is the identity, so that $A^{**} = A$.
\end{definition}

In a pivotal category, left and right duals coincide (up to the pivotal isomorphism), and the graphical calculus acquires full rotational invariance within the plane: diagrams can be bent and turned without changing the morphism they represent. This is the diagrammatic setting natural for planar arc diagrams.

\section{The category \texorpdfstring{$\Ddna$ (DNA diagrams)}{D-DNA (DNA diagrams)}}\label{sec:category}

We now construct the category. The reader should keep in mind two pictures simultaneously: the algebraic data (objects, morphisms, composition) and their diagrammatic representations as typed arc diagrams.

\subsection{Objects}

The objects of $\Ddna$ are the words in the free monoid $\Sig^*$. The tensor product on objects is concatenation:
\[
x \otimes y := xy,
\]
and the monoidal unit is the empty word $\eps$. This is strictly associative and unital.

\subsection{Duality}

The dual of a word $w = b_1 \cdots b_n$ is its reverse complement:
\[
w\dv = \comp{b_n}\, \comp{b_{n-1}} \cdots \comp{b_1}.
\]
Three properties are immediate from the definitions and are worth recording:
\begin{enumerate}
  \item $\eps\dv = \eps$.
  \item $(uv)\dv = v\dv\, u\dv$ for all words $u, v$.
  \item $(w\dv)\dv = w$ for every word $w$.
\end{enumerate}
Property~(3) implies that the double dual is the identity on objects.

\subsection{Morphisms}\label{subsec:morphisms}

A morphism $f \colon x \to y$ in $\Ddna$ is an isotopy class of typed noncrossing planar partial matchings, drawn in a rectangle with the bases of $x$ arranged along the top boundary (the source) and the bases of $y$ arranged along the bottom boundary (the target). The matching consists of three components:

\begin{enumerate}
  \item \textbf{Through-wires}: wires connecting a position of $x$ to a position of $y$, representing a base that passes from source to target.
  \item \textbf{Source arcs}: noncrossing arcs among positions of $x$, representing internal base pairings on the source side.
  \item \textbf{Target arcs}: noncrossing arcs among positions of $y$, representing internal base pairings on the target side.
\end{enumerate}

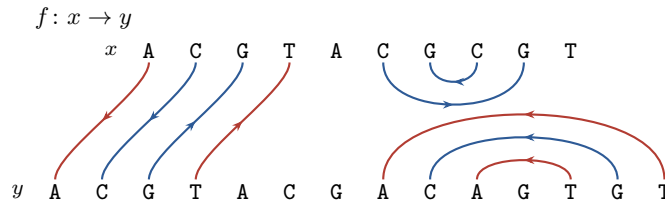
\begin{center}
\begin{tikzpicture}[x=0.62cm, y=0.62cm]
  \node[rigid node, above right] at (0.50,3.40) {\small $f\colon x\to y$};
  \node[base node] (s0) at (3.00,3.00) {A};
  \node[base node] (s1) at (4.00,3.00) {C};
  \node[base node] (s2) at (5.00,3.00) {G};
  \node[base node] (s3) at (6.00,3.00) {T};
  \node[base node] (s4) at (7.00,3.00) {A};
  \node[base node] (s5) at (8.00,3.00) {C};
  \node[base node] (s6) at (9.00,3.00) {G};
  \node[base node] (s7) at (10.00,3.00) {C};
  \node[base node] (s8) at (11.00,3.00) {G};
  \node[base node] (s9) at (12.00,3.00) {T};
  \node[rigid node, left] at (2.40,3.00) {\footnotesize$x$};
  \node[base node] (t0) at (1.00,0.00) {A};
  \node[base node] (t1) at (2.00,0.00) {C};
  \node[base node] (t2) at (3.00,0.00) {G};
  \node[base node] (t3) at (4.00,0.00) {T};
  \node[base node] (t4) at (5.00,0.00) {A};
  \node[base node] (t5) at (6.00,0.00) {C};
  \node[base node] (t6) at (7.00,0.00) {G};
  \node[base node] (t7) at (8.00,0.00) {A};
  \node[base node] (t8) at (9.00,0.00) {C};
  \node[base node] (t9) at (10.00,0.00) {A};
  \node[base node] (t10) at (11.00,0.00) {G};
  \node[base node] (t11) at (12.00,0.00) {T};
  \node[base node] (t12) at (13.00,0.00) {G};
  \node[base node] (t13) at (14.00,0.00) {T};
  \node[rigid node, left] at (0.40,0.00) {\footnotesize$y$};
  \draw[cg arc] (s5.south) .. controls +(0,-1.20) and +(0,-1.20) .. (s8.south);
  \draw[cg arc rev] (s6.south) .. controls +(0,-0.55) and +(0,-0.55) .. (s7.south);
  \draw[at arc rev] (t7.north) .. controls +(0,1.85) and +(0,1.85) .. (t13.north);
  \draw[cg arc rev] (t8.north) .. controls +(0,1.20) and +(0,1.20) .. (t12.north);
  \draw[at arc rev] (t9.north) .. controls +(0,0.55) and +(0,0.55) .. (t11.north);
  \draw[at wire] (s0.south) .. controls (3.00,2.16) and (1.00,0.84) .. (t0.north);
  \draw[cg wire] (s1.south) .. controls (4.00,2.16) and (2.00,0.84) .. (t1.north);
  \draw[cg wire rev] (s2.south) .. controls (5.00,2.16) and (3.00,0.84) .. (t2.north);
  \draw[at wire rev] (s3.south) .. controls (6.00,2.16) and (4.00,0.84) .. (t3.north);
\end{tikzpicture}
\end{center}

The global constraints are:
\begin{itemize}
  \item \emph{Degree}: every position participates in at most one edge or arc.
  \item \emph{Planarity}: the union of all strands and arcs is noncrossing in the rectangle.
  \item \emph{Pairing and wire conventions}: every arc connects Watson--Crick complementary bases, while through-strands are identity wires (the same base on source and target). We draw A--T pairs in red and C--G pairs in blue; A and C wires go down, while T and G wires go up.
\end{itemize}

Two such diagrams represent the same morphism if they are related by planar isotopy within the rectangle, keeping the boundary points fixed.


\subsection{Generalized elements: secondary structures as terms}

A morphism $s \colon \eps \to w$ is a generalized element of the object $w$: it is a diagram with no source boundary and the bases of $w$ on the target boundary. Concretely, it consists entirely of target arcs---that is, it is a noncrossing arc diagram on the positions of $w$ with complementary pairings. This is exactly a non-pseudoknotted secondary structure on $w$.\footnote{From the internal type-theoretic viewpoint of categorical logic, objects are types (here: DNA words), morphisms are terms/proofs/programs (here: DNA diagrams), and maps $\eps \to w$ are closed terms of type $w$, i.e., generalized elements of $w$.}

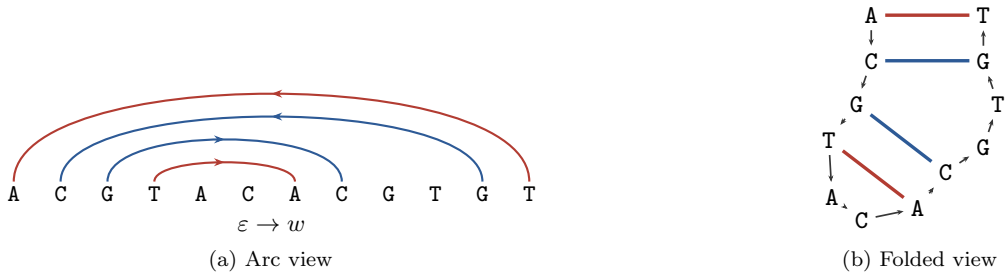
\begin{center}
\begin{minipage}[t]{0.48\textwidth}
\centering
\begin{tikzpicture}[x=0.62cm, y=0.62cm]
  \node[base node] (t0) at (1.00,0.00) {A};
  \node[base node] (t1) at (2.00,0.00) {C};
  \node[base node] (t2) at (3.00,0.00) {G};
  \node[base node] (t3) at (4.00,0.00) {T};
  \node[base node] (t4) at (5.00,0.00) {A};
  \node[base node] (t5) at (6.00,0.00) {C};
  \node[base node] (t6) at (7.00,0.00) {A};
  \node[base node] (t7) at (8.00,0.00) {C};
  \node[base node] (t8) at (9.00,0.00) {G};
  \node[base node] (t9) at (10.00,0.00) {T};
  \node[base node] (t10) at (11.00,0.00) {G};
  \node[base node] (t11) at (12.00,0.00) {T};
  \draw[at arc rev] (t0.north) .. controls +(0,2.50) and +(0,2.50) .. (t11.north);
  \draw[cg arc rev] (t1.north) .. controls +(0,1.85) and +(0,1.85) .. (t10.north);
  \draw[cg arc] (t2.north) .. controls +(0,1.20) and +(0,1.20) .. (t7.north);
  \draw[at arc] (t3.north) .. controls +(0,0.55) and +(0,0.55) .. (t6.north);
  \node[rigid node] at (6.5,-0.7) {$\eps \to w$};
\end{tikzpicture}

\smallskip
\footnotesize (a) Arc view
\end{minipage}\hfill
\begin{minipage}[t]{0.48\textwidth}
\centering
\resizebox{0.35\textwidth}{!}{
\begin{tikzpicture}[x=0.58cm, y=0.58cm]
  \node[base node, font=\tiny\ttfamily] (b0) at (-0.800,-0.000) {A};
  \node[base node, font=\tiny\ttfamily] (b1) at (-0.800,-0.650) {C};
  \node[base node, font=\tiny\ttfamily] (b2) at (-1.009,-1.265) {G};
  \node[base node, font=\tiny\ttfamily] (b3) at (-1.406,-1.781) {T};
  \node[base node, font=\tiny\ttfamily] (b4) at (-1.359,-2.595) {A};
  \node[base node, font=\tiny\ttfamily] (b5) at (-0.937,-2.920) {C};
  \node[base node, font=\tiny\ttfamily] (b6) at (-0.137,-2.756) {A};
  \node[base node, font=\tiny\ttfamily] (b7) at (0.259,-2.241) {C};
  \node[base node, font=\tiny\ttfamily] (b8) at (0.800,-1.881) {G};
  \node[base node, font=\tiny\ttfamily] (b9) at (1.009,-1.265) {T};
  \node[base node, font=\tiny\ttfamily] (b10) at (0.800,-0.650) {G};
  \node[base node, font=\tiny\ttfamily] (b11) at (0.800,0.000) {T};
  \draw[-{Stealth[length=1.6pt,width=1.2pt]}, dnaWire, thin] (b0) -- (b1);
  \draw[-{Stealth[length=1.6pt,width=1.2pt]}, dnaWire, thin] (b1) -- (b2);
  \draw[-{Stealth[length=1.6pt,width=1.2pt]}, dnaWire, thin] (b2) -- (b3);
  \draw[-{Stealth[length=1.6pt,width=1.2pt]}, dnaWire, thin] (b3) -- (b4);
  \draw[-{Stealth[length=1.6pt,width=1.2pt]}, dnaWire, thin] (b4) -- (b5);
  \draw[-{Stealth[length=1.6pt,width=1.2pt]}, dnaWire, thin] (b5) -- (b6);
  \draw[-{Stealth[length=1.6pt,width=1.2pt]}, dnaWire, thin] (b6) -- (b7);
  \draw[-{Stealth[length=1.6pt,width=1.2pt]}, dnaWire, thin] (b7) -- (b8);
  \draw[-{Stealth[length=1.6pt,width=1.2pt]}, dnaWire, thin] (b8) -- (b9);
  \draw[-{Stealth[length=1.6pt,width=1.2pt]}, dnaWire, thin] (b9) -- (b10);
  \draw[-{Stealth[length=1.6pt,width=1.2pt]}, dnaWire, thin] (b10) -- (b11);
  \draw[dnaAT, thick] (b0) -- (b11);
  \draw[dnaCG, thick] (b1) -- (b10);
  \draw[dnaCG, thick] (b2) -- (b7);
  \draw[dnaAT, thick] (b3) -- (b6);
\end{tikzpicture}}

\smallskip
\footnotesize (b) Folded view
\end{minipage}
\end{center}

In the folded view, backbone arrows indicate the $5' \to 3'$ direction (left to right, matching the arc-view ordering), and arcs are rendered as base-pair bonds.\footnote{All diagrams in this paper are schematic and intended for illustration only: the relative lengths of backbone bonds and Watson--Crick pair bonds are not to scale, and real hairpins have a minimum loop size imposed by geometric and physical constraints.}

Conversely, every non-pseudoknotted secondary structure on $w$ defines a morphism $\eps \to w$ in $\Ddna$. Equivalently, $\Hom(\eps, w)$ is the set of generalized elements (closed terms) of type $w$, and thus catalogs the secondary structures of $w$.

\subsection{Identity and composition}

The identity morphism $\id_w \colon w \to w$ is the straight-through matching: position $i$ on the source connects to position $i$ on the target, with no internal arcs.

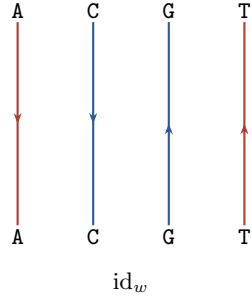
\begin{center}
\begin{tikzpicture}[x=1cm, y=1cm]
  \node[base node] (s0) at (1.00,3.00) {A};
  \node[base node] (s1) at (2.00,3.00) {C};
  \node[base node] (s2) at (3.00,3.00) {G};
  \node[base node] (s3) at (4.00,3.00) {T};
  \node[base node] (t0) at (1.00,0.00) {A};
  \node[base node] (t1) at (2.00,0.00) {C};
  \node[base node] (t2) at (3.00,0.00) {G};
  \node[base node] (t3) at (4.00,0.00) {T};
  \draw[at wire] (s0.south) .. controls (1.00,2.16) and (1.00,0.84) .. (t0.north);
  \draw[cg wire] (s1.south) .. controls (2.00,2.16) and (2.00,0.84) .. (t1.north);
  \draw[cg wire rev] (s2.south) .. controls (3.00,2.16) and (3.00,0.84) .. (t2.north);
  \draw[at wire rev] (s3.south) .. controls (4.00,2.16) and (4.00,0.84) .. (t3.north);
  \node[rigid node] at (2.5,-0.6) {$\id_w$};
\end{tikzpicture}
\end{center}

Given $f \colon x \to y$ and $g \colon y \to z$, the composite $g \circ f \colon x \to z$ is formed by vertical stacking: place the rectangle for $f$ on top and the rectangle for $g$ below, gluing along the shared $y$-boundary. After gluing, simplify by planar isotopy and cap-cup cancellations (the snake relations). Associativity of composition follows from the topological nature of stacking.

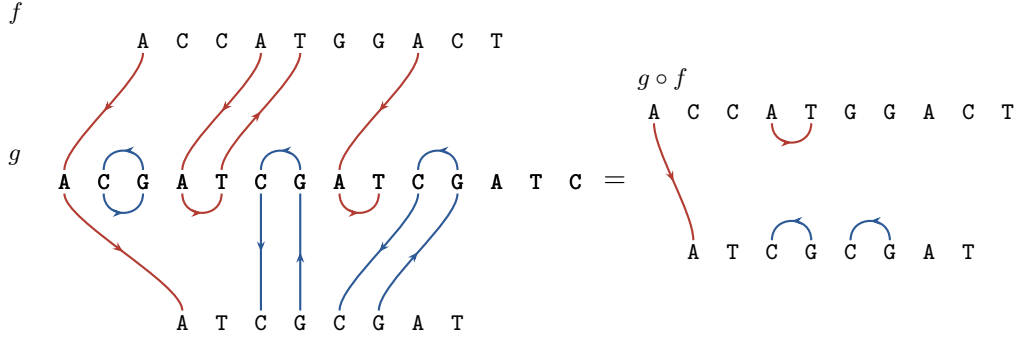
\begin{center}
\begin{tikzpicture}[x=0.52cm, y=0.62cm]
  \node[rigid node, above right] at (-0.50,6.30) {\small $f$};
  \node[rigid node, above right] at (-0.50,3.30) {\small $g$};
  \node[base node] (m0s0) at (3.00,6.00) {A};
  \node[base node] (m0s1) at (4.00,6.00) {C};
  \node[base node] (m0s2) at (5.00,6.00) {C};
  \node[base node] (m0s3) at (6.00,6.00) {A};
  \node[base node] (m0s4) at (7.00,6.00) {T};
  \node[base node] (m0s5) at (8.00,6.00) {G};
  \node[base node] (m0s6) at (9.00,6.00) {G};
  \node[base node] (m0s7) at (10.00,6.00) {A};
  \node[base node] (m0s8) at (11.00,6.00) {C};
  \node[base node] (m0s9) at (12.00,6.00) {T};
  \node[base node] (m0t0) at (1.00,3.00) {A};
  \node[base node] (m0t1) at (2.00,3.00) {C};
  \node[base node] (m0t2) at (3.00,3.00) {G};
  \node[base node] (m0t3) at (4.00,3.00) {A};
  \node[base node] (m0t4) at (5.00,3.00) {T};
  \node[base node] (m0t5) at (6.00,3.00) {C};
  \node[base node] (m0t6) at (7.00,3.00) {G};
  \node[base node] (m0t7) at (8.00,3.00) {A};
  \node[base node] (m0t8) at (9.00,3.00) {T};
  \node[base node] (m0t9) at (10.00,3.00) {C};
  \node[base node] (m0t10) at (11.00,3.00) {G};
  \node[base node] (m0t11) at (12.00,3.00) {A};
  \node[base node] (m0t12) at (13.00,3.00) {T};
  \node[base node] (m0t13) at (14.00,3.00) {C};
  \draw[cg arc rev] (m0t9.north) .. controls +(0,0.55) and +(0,0.55) .. (m0t10.north);
  \draw[cg arc rev] (m0t5.north) .. controls +(0,0.55) and +(0,0.55) .. (m0t6.north);
  \draw[cg arc rev] (m0t1.north) .. controls +(0,0.55) and +(0,0.55) .. (m0t2.north);
  \draw[at wire] (m0s0.south) .. controls (3.00,5.16) and (1.00,3.84) .. (m0t0.north);
  \draw[at wire] (m0s3.south) .. controls (6.00,5.16) and (4.00,3.84) .. (m0t3.north);
  \draw[at wire rev] (m0s4.south) .. controls (7.00,5.16) and (5.00,3.84) .. (m0t4.north);
  \draw[at wire] (m0s7.south) .. controls (10.00,5.16) and (8.00,3.84) .. (m0t7.north);
  \node[base node] (m1s0) at (1.00,3.00) {A};
  \node[base node] (m1s1) at (2.00,3.00) {C};
  \node[base node] (m1s2) at (3.00,3.00) {G};
  \node[base node] (m1s3) at (4.00,3.00) {A};
  \node[base node] (m1s4) at (5.00,3.00) {T};
  \node[base node] (m1s5) at (6.00,3.00) {C};
  \node[base node] (m1s6) at (7.00,3.00) {G};
  \node[base node] (m1s7) at (8.00,3.00) {A};
  \node[base node] (m1s8) at (9.00,3.00) {T};
  \node[base node] (m1s9) at (10.00,3.00) {C};
  \node[base node] (m1s10) at (11.00,3.00) {G};
  \node[base node] (m1s11) at (12.00,3.00) {A};
  \node[base node] (m1s12) at (13.00,3.00) {T};
  \node[base node] (m1s13) at (14.00,3.00) {C};
  \node[base node] (m1t0) at (4.00,0.00) {A};
  \node[base node] (m1t1) at (5.00,0.00) {T};
  \node[base node] (m1t2) at (6.00,0.00) {C};
  \node[base node] (m1t3) at (7.00,0.00) {G};
  \node[base node] (m1t4) at (8.00,0.00) {C};
  \node[base node] (m1t5) at (9.00,0.00) {G};
  \node[base node] (m1t6) at (10.00,0.00) {A};
  \node[base node] (m1t7) at (11.00,0.00) {T};
  \draw[cg arc] (m1s1.south) .. controls +(0,-0.55) and +(0,-0.55) .. (m1s2.south);
  \draw[at arc] (m1s3.south) .. controls +(0,-0.55) and +(0,-0.55) .. (m1s4.south);
  \draw[at arc] (m1s7.south) .. controls +(0,-0.55) and +(0,-0.55) .. (m1s8.south);
  \draw[at wire] (m1s0.south) .. controls (1.00,2.16) and (4.00,0.84) .. (m1t0.north);
  \draw[cg wire] (m1s5.south) .. controls (6.00,2.16) and (6.00,0.84) .. (m1t2.north);
  \draw[cg wire rev] (m1s6.south) .. controls (7.00,2.16) and (7.00,0.84) .. (m1t3.north);
  \draw[cg wire] (m1s9.south) .. controls (10.00,2.16) and (8.00,0.84) .. (m1t4.north);
  \draw[cg wire rev] (m1s10.south) .. controls (11.00,2.16) and (9.00,0.84) .. (m1t5.north);
  \node[rigid eq] at (15.0,3.0) {$=$};
  \node[rigid node, above right] at (15.50,4.90) {\small $g\circ f$};
  \node[base node] (cs0) at (16.00,4.50) {A};
  \node[base node] (cs1) at (17.00,4.50) {C};
  \node[base node] (cs2) at (18.00,4.50) {C};
  \node[base node] (cs3) at (19.00,4.50) {A};
  \node[base node] (cs4) at (20.00,4.50) {T};
  \node[base node] (cs5) at (21.00,4.50) {G};
  \node[base node] (cs6) at (22.00,4.50) {G};
  \node[base node] (cs7) at (23.00,4.50) {A};
  \node[base node] (cs8) at (24.00,4.50) {C};
  \node[base node] (cs9) at (25.00,4.50) {T};
  \node[base node] (ct0) at (17.00,1.50) {A};
  \node[base node] (ct1) at (18.00,1.50) {T};
  \node[base node] (ct2) at (19.00,1.50) {C};
  \node[base node] (ct3) at (20.00,1.50) {G};
  \node[base node] (ct4) at (21.00,1.50) {C};
  \node[base node] (ct5) at (22.00,1.50) {G};
  \node[base node] (ct6) at (23.00,1.50) {A};
  \node[base node] (ct7) at (24.00,1.50) {T};
  \draw[at arc] (cs3.south) .. controls +(0,-0.55) and +(0,-0.55) .. (cs4.south);
  \draw[cg arc rev] (ct4.north) .. controls +(0,0.55) and +(0,0.55) .. (ct5.north);
  \draw[cg arc rev] (ct2.north) .. controls +(0,0.55) and +(0,0.55) .. (ct3.north);
  \draw[at wire] (cs0.south) .. controls (16.00,3.66) and (17.00,2.34) .. (ct0.north);
\end{tikzpicture}
\end{center}

Under composition, closed loops---connected components with no boundary endpoints---may arise. We adopt the \emph{normalized} convention throughout: every closed loop equals $1_\eps$ and is erased. Other conventions (e.g.\ loops weighted by scalars $\delta_{AT}$, $\delta_{CG}$) yield valid alternative categories; see Section~\ref{sec:remarks}.

\subsection{Tensor product on morphisms}

The tensor product $f \otimes g$ of morphisms $f \colon x_1 \to y_1$ and $g \colon x_2 \to y_2$ is their horizontal juxtaposition: the diagram for $f$ is placed to the left of the diagram for $g$, giving a morphism $x_1 x_2 \to y_1 y_2$.

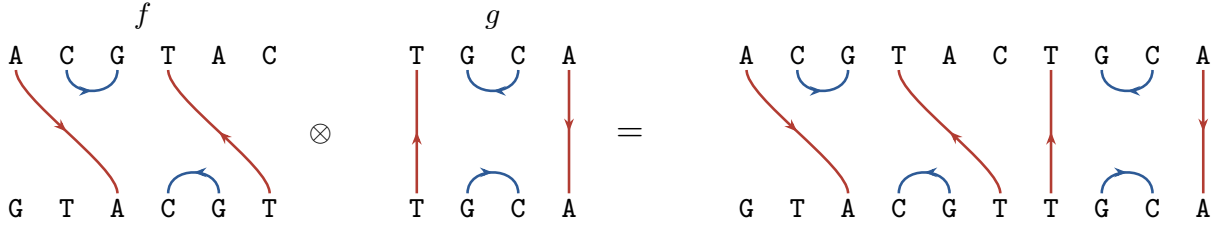
\begin{center}
\resizebox{0.98\textwidth}{!}{
\begin{tikzpicture}[x=0.55cm, y=0.55cm]
  \node[base node] (tfs0) at (1.00,3.00) {A};
  \node[base node] (tfs1) at (2.00,3.00) {C};
  \node[base node] (tfs2) at (3.00,3.00) {G};
  \node[base node] (tfs3) at (4.00,3.00) {T};
  \node[base node] (tfs4) at (5.00,3.00) {A};
  \node[base node] (tfs5) at (6.00,3.00) {C};
  \node[base node] (tft0) at (1.00,0.00) {G};
  \node[base node] (tft1) at (2.00,0.00) {T};
  \node[base node] (tft2) at (3.00,0.00) {A};
  \node[base node] (tft3) at (4.00,0.00) {C};
  \node[base node] (tft4) at (5.00,0.00) {G};
  \node[base node] (tft5) at (6.00,0.00) {T};
  \draw[cg arc] (tfs1.south) .. controls +(0,-0.55) and +(0,-0.55) .. (tfs2.south);
  \draw[cg arc rev] (tft3.north) .. controls +(0,0.55) and +(0,0.55) .. (tft4.north);
  \draw[at wire] (tfs0.south) .. controls (1.00,2.16) and (3.00,0.84) .. (tft2.north);
  \draw[at wire rev] (tfs3.south) .. controls (4.00,2.16) and (6.00,0.84) .. (tft5.north);
  \node[base node] (tgs0) at (8.90,3.00) {T};
  \node[base node] (tgs1) at (9.90,3.00) {G};
  \node[base node] (tgs2) at (10.90,3.00) {C};
  \node[base node] (tgs3) at (11.90,3.00) {A};
  \node[base node] (tgt0) at (8.90,0.00) {T};
  \node[base node] (tgt1) at (9.90,0.00) {G};
  \node[base node] (tgt2) at (10.90,0.00) {C};
  \node[base node] (tgt3) at (11.90,0.00) {A};
  \draw[cg arc rev] (tgs1.south) .. controls +(0,-0.55) and +(0,-0.55) .. (tgs2.south);
  \draw[cg arc] (tgt1.north) .. controls +(0,0.55) and +(0,0.55) .. (tgt2.north);
  \draw[at wire rev] (tgs0.south) .. controls (8.90,2.16) and (8.90,0.84) .. (tgt0.north);
  \draw[at wire] (tgs3.south) .. controls (11.90,2.16) and (11.90,0.84) .. (tgt3.north);
  \node[base node] (ths0) at (15.40,3.00) {A};
  \node[base node] (ths1) at (16.40,3.00) {C};
  \node[base node] (ths2) at (17.40,3.00) {G};
  \node[base node] (ths3) at (18.40,3.00) {T};
  \node[base node] (ths4) at (19.40,3.00) {A};
  \node[base node] (ths5) at (20.40,3.00) {C};
  \node[base node] (ths6) at (21.40,3.00) {T};
  \node[base node] (ths7) at (22.40,3.00) {G};
  \node[base node] (ths8) at (23.40,3.00) {C};
  \node[base node] (ths9) at (24.40,3.00) {A};
  \node[base node] (tht0) at (15.40,0.00) {G};
  \node[base node] (tht1) at (16.40,0.00) {T};
  \node[base node] (tht2) at (17.40,0.00) {A};
  \node[base node] (tht3) at (18.40,0.00) {C};
  \node[base node] (tht4) at (19.40,0.00) {G};
  \node[base node] (tht5) at (20.40,0.00) {T};
  \node[base node] (tht6) at (21.40,0.00) {T};
  \node[base node] (tht7) at (22.40,0.00) {G};
  \node[base node] (tht8) at (23.40,0.00) {C};
  \node[base node] (tht9) at (24.40,0.00) {A};
  \draw[cg arc] (ths1.south) .. controls +(0,-0.55) and +(0,-0.55) .. (ths2.south);
  \draw[cg arc rev] (ths7.south) .. controls +(0,-0.55) and +(0,-0.55) .. (ths8.south);
  \draw[cg arc] (tht7.north) .. controls +(0,0.55) and +(0,0.55) .. (tht8.north);
  \draw[cg arc rev] (tht3.north) .. controls +(0,0.55) and +(0,0.55) .. (tht4.north);
  \draw[at wire] (ths0.south) .. controls (15.40,2.16) and (17.40,0.84) .. (tht2.north);
  \draw[at wire rev] (ths3.south) .. controls (18.40,2.16) and (20.40,0.84) .. (tht5.north);
  \draw[at wire rev] (ths6.south) .. controls (21.40,2.16) and (21.40,0.84) .. (tht6.north);
  \draw[at wire] (ths9.south) .. controls (24.40,2.16) and (24.40,0.84) .. (tht9.north);
  \node[rigid node, above] at (3.5,3.45) {\small $f$};
  \node[rigid node, above] at (10.4,3.45) {\small $g$};
  \node[rigid node] at (7.0,1.45) {$\otimes$};
  \node[rigid eq] at (13.1,1.45) {$=$};
\end{tikzpicture}}
\end{center}

\subsection{Evaluation, coevaluation, and rigidity}

For each word $w = b_1 \cdots b_n$, we define:
\begin{itemize}
  \item The \emph{coevaluation} (cap):
  \[
  \coev_w \colon \eps \to w\dv \otimes w.
  \]
  This is the noncrossing diagram that pairs each base $b_i$ in $w$ with the corresponding complementary base $\comp{b_i}$ in $w\dv$, forming $n$ nested arcs. It represents the canonical duplex pairing.

  \item The \emph{evaluation} (cup):
  \[
  \ev_w \colon w \otimes w\dv \to \eps.
  \]
  This is the vertically reflected version: arcs on the source side pair each position of $w$ with its complement in $w\dv$.
\end{itemize}

\begin{center}
\begin{minipage}[t]{0.48\textwidth}
\centering
\begin{tikzpicture}[
  x=1cm, y=1cm,
  at arc/.style={dnaAT, thick, ->},
  at arc rev/.style={dnaAT, thick, <-},
  cg arc/.style={dnaCG, thick, ->},
  cg arc rev/.style={dnaCG, thick, <-},
]
  \node[base node] (s0) at (1.00,2.40) {A};
  \node[base node] (s1) at (2.00,2.40) {C};
  \node[base node] (s2) at (3.00,2.40) {G};
  \node[base node] (s3) at (4.00,2.40) {C};
  \node[base node] (s4) at (5.00,2.40) {G};
  \node[base node] (s5) at (6.00,2.40) {T};
  \draw[at arc] (s0.south) .. controls +(0,-1.85) and +(0,-1.85) .. (s5.south);
  \draw[cg arc] (s1.south) .. controls +(0,-1.20) and +(0,-1.20) .. (s4.south);
  \draw[cg arc rev] (s2.south) .. controls +(0,-0.55) and +(0,-0.55) .. (s3.south);
  \node[rigid node] at (2,2.85) {\footnotesize$w = \texttt{ACG}$};
  \node[rigid node] at (5,2.85) {\footnotesize$w^\vee = \texttt{CGT}$};
  \node[rigid node] at (3.5,-0.80) {$\ev_w \colon w \otimes w^\vee \to \eps$};
\end{tikzpicture}
\end{minipage}\hfill
\begin{minipage}[t]{0.48\textwidth}
\centering
\begin{tikzpicture}[
  x=1cm, y=1cm,
  at arc/.style={dnaAT, thick, ->},
  at arc rev/.style={dnaAT, thick, <-},
  cg arc/.style={dnaCG, thick, ->},
  cg arc rev/.style={dnaCG, thick, <-},
]
  \node[base node] (t0) at (1.00,0.00) {C};
  \node[base node] (t1) at (2.00,0.00) {G};
  \node[base node] (t2) at (3.00,0.00) {T};
  \node[base node] (t3) at (4.00,0.00) {A};
  \node[base node] (t4) at (5.00,0.00) {C};
  \node[base node] (t5) at (6.00,0.00) {G};
  \draw[cg arc rev] (t0.north) .. controls +(0,1.85) and +(0,1.85) .. (t5.north);
  \draw[cg arc] (t1.north) .. controls +(0,1.20) and +(0,1.20) .. (t4.north);
  \draw[at arc] (t2.north) .. controls +(0,0.55) and +(0,0.55) .. (t3.north);
  \node[rigid node] at (2,-0.55) {\footnotesize$w^\vee = \texttt{CGT}$};
  \node[rigid node] at (5,-0.55) {\footnotesize$w = \texttt{ACG}$};
  \node[rigid node] at (3.5,-1.10) {$\coev_w \colon \eps \to w^\vee \otimes w$};
\end{tikzpicture}
\end{minipage}
\end{center}

\begin{proposition}\label{prop:snake}
The snake identities hold in $\Ddna$:
\[
(\ev_w \otimes \id_w)\circ(\id_w \otimes \coev_w) = \id_w,
\]
\[
(\id_{w\dv} \otimes \ev_w)\circ(\coev_w \otimes \id_{w\dv}) = \id_{w\dv}.
\]
\end{proposition}

\begin{proof}
Each identity asserts that a zig-zag of arcs---a cup followed by a cap on the same strand---can be straightened to a single through-strand. In the diagrammatic picture, the left-hand side of the first identity is a wire for $w$ that bends up to create a cup (introducing $w\dv$ and $w$), then bends back down via a cap (annihilating $w$ and $w\dv$). By planar isotopy, this zig-zag straightens to the identity wire for $w$. The second identity is analogous.
\end{proof}

\begin{center}
\begin{tikzpicture}[x=0.72cm, y=0.72cm]
  \node[base node] (m0s0) at (1.00,6.00) {A};
  \node[base node] (m0s1) at (2.00,6.00) {C};
  \node[base node] (m0s2) at (3.00,6.00) {G};
  \node[base node] (m0t0) at (1.00,3.00) {A};
  \node[base node] (m0t1) at (2.00,3.00) {C};
  \node[base node] (m0t2) at (3.00,3.00) {G};
  \node[base node] (m0t3) at (4.00,3.00) {C};
  \node[base node] (m0t4) at (5.00,3.00) {G};
  \node[base node] (m0t5) at (6.00,3.00) {T};
  \node[base node] (m0t6) at (7.00,3.00) {A};
  \node[base node] (m0t7) at (8.00,3.00) {C};
  \node[base node] (m0t8) at (9.00,3.00) {G};
  \draw[cg arc rev] (m0t3.north) .. controls +(0,1.85) and +(0,1.85) .. (m0t8.north);
  \draw[cg arc] (m0t4.north) .. controls +(0,1.20) and +(0,1.20) .. (m0t7.north);
  \draw[at arc] (m0t5.north) .. controls +(0,0.55) and +(0,0.55) .. (m0t6.north);
  \draw[at wire] (m0s0.south) .. controls (1.00,5.16) and (1.00,3.84) .. (m0t0.north);
  \draw[cg wire] (m0s1.south) .. controls (2.00,5.16) and (2.00,3.84) .. (m0t1.north);
  \draw[cg wire rev] (m0s2.south) .. controls (3.00,5.16) and (3.00,3.84) .. (m0t2.north);
  \node[base node] (m1s0) at (1.00,3.00) {A};
  \node[base node] (m1s1) at (2.00,3.00) {C};
  \node[base node] (m1s2) at (3.00,3.00) {G};
  \node[base node] (m1s3) at (4.00,3.00) {C};
  \node[base node] (m1s4) at (5.00,3.00) {G};
  \node[base node] (m1s5) at (6.00,3.00) {T};
  \node[base node] (m1s6) at (7.00,3.00) {A};
  \node[base node] (m1s7) at (8.00,3.00) {C};
  \node[base node] (m1s8) at (9.00,3.00) {G};
  \node[base node] (m1t0) at (7.00,0.00) {A};
  \node[base node] (m1t1) at (8.00,0.00) {C};
  \node[base node] (m1t2) at (9.00,0.00) {G};
  \draw[at arc] (m1s0.south) .. controls +(0,-1.85) and +(0,-1.85) .. (m1s5.south);
  \draw[cg arc] (m1s1.south) .. controls +(0,-1.20) and +(0,-1.20) .. (m1s4.south);
  \draw[cg arc rev] (m1s2.south) .. controls +(0,-0.55) and +(0,-0.55) .. (m1s3.south);
  \draw[at wire] (m1s6.south) .. controls (7.00,2.16) and (7.00,0.84) .. (m1t0.north);
  \draw[cg wire] (m1s7.south) .. controls (8.00,2.16) and (8.00,0.84) .. (m1t1.north);
  \draw[cg wire rev] (m1s8.south) .. controls (9.00,2.16) and (9.00,0.84) .. (m1t2.north);
  \node[rigid eq] at (10.5,3.0) {$=$};
  \node[base node] (snake_rhss0) at (13.00,4.50) {A};
  \node[base node] (snake_rhss1) at (14.00,4.50) {C};
  \node[base node] (snake_rhss2) at (15.00,4.50) {G};
  \node[base node] (snake_rhst0) at (13.00,1.50) {A};
  \node[base node] (snake_rhst1) at (14.00,1.50) {C};
  \node[base node] (snake_rhst2) at (15.00,1.50) {G};
  \draw[at wire] (snake_rhss0.south) .. controls (13.00,3.66) and (13.00,2.34) .. (snake_rhst0.north);
  \draw[cg wire] (snake_rhss1.south) .. controls (14.00,3.66) and (14.00,2.34) .. (snake_rhst1.north);
  \draw[cg wire rev] (snake_rhss2.south) .. controls (15.00,3.66) and (15.00,2.34) .. (snake_rhst2.north);
  \node[rigid node] at (8.7,-0.85) {$\left(\ev_w \otimes \id_w\right)\circ\left(\id_w \otimes \coev_w\right)=\id_w$};
\end{tikzpicture}
\end{center}
\begin{center}
\begin{tikzpicture}[x=0.72cm, y=0.72cm]
  \node[base node] (m0s0) at (7.00,6.00) {C};
  \node[base node] (m0s1) at (8.00,6.00) {G};
  \node[base node] (m0s2) at (9.00,6.00) {T};
  \node[base node] (m0t0) at (1.00,3.00) {C};
  \node[base node] (m0t1) at (2.00,3.00) {G};
  \node[base node] (m0t2) at (3.00,3.00) {T};
  \node[base node] (m0t3) at (4.00,3.00) {A};
  \node[base node] (m0t4) at (5.00,3.00) {C};
  \node[base node] (m0t5) at (6.00,3.00) {G};
  \node[base node] (m0t6) at (7.00,3.00) {C};
  \node[base node] (m0t7) at (8.00,3.00) {G};
  \node[base node] (m0t8) at (9.00,3.00) {T};
  \draw[cg arc rev] (m0t0.north) .. controls +(0,1.85) and +(0,1.85) .. (m0t5.north);
  \draw[cg arc] (m0t1.north) .. controls +(0,1.20) and +(0,1.20) .. (m0t4.north);
  \draw[at arc] (m0t2.north) .. controls +(0,0.55) and +(0,0.55) .. (m0t3.north);
  \draw[cg wire] (m0s0.south) .. controls (7.00,5.16) and (7.00,3.84) .. (m0t6.north);
  \draw[cg wire rev] (m0s1.south) .. controls (8.00,5.16) and (8.00,3.84) .. (m0t7.north);
  \draw[at wire rev] (m0s2.south) .. controls (9.00,5.16) and (9.00,3.84) .. (m0t8.north);
  \node[base node] (m1s0) at (1.00,3.00) {C};
  \node[base node] (m1s1) at (2.00,3.00) {G};
  \node[base node] (m1s2) at (3.00,3.00) {T};
  \node[base node] (m1s3) at (4.00,3.00) {A};
  \node[base node] (m1s4) at (5.00,3.00) {C};
  \node[base node] (m1s5) at (6.00,3.00) {G};
  \node[base node] (m1s6) at (7.00,3.00) {C};
  \node[base node] (m1s7) at (8.00,3.00) {G};
  \node[base node] (m1s8) at (9.00,3.00) {T};
  \node[base node] (m1t0) at (1.00,0.00) {C};
  \node[base node] (m1t1) at (2.00,0.00) {G};
  \node[base node] (m1t2) at (3.00,0.00) {T};
  \draw[at arc] (m1s3.south) .. controls +(0,-1.85) and +(0,-1.85) .. (m1s8.south);
  \draw[cg arc] (m1s4.south) .. controls +(0,-1.20) and +(0,-1.20) .. (m1s7.south);
  \draw[cg arc rev] (m1s5.south) .. controls +(0,-0.55) and +(0,-0.55) .. (m1s6.south);
  \draw[cg wire] (m1s0.south) .. controls (1.00,2.16) and (1.00,0.84) .. (m1t0.north);
  \draw[cg wire rev] (m1s1.south) .. controls (2.00,2.16) and (2.00,0.84) .. (m1t1.north);
  \draw[at wire rev] (m1s2.south) .. controls (3.00,2.16) and (3.00,0.84) .. (m1t2.north);
  \node[rigid eq] at (10.5,3.0) {$=$};
  \node[base node] (snake_rhss0) at (13.00,4.50) {C};
  \node[base node] (snake_rhss1) at (14.00,4.50) {G};
  \node[base node] (snake_rhss2) at (15.00,4.50) {T};
  \node[base node] (snake_rhst0) at (13.00,1.50) {C};
  \node[base node] (snake_rhst1) at (14.00,1.50) {G};
  \node[base node] (snake_rhst2) at (15.00,1.50) {T};
  \draw[cg wire] (snake_rhss0.south) .. controls (13.00,3.66) and (13.00,2.34) .. (snake_rhst0.north);
  \draw[cg wire rev] (snake_rhss1.south) .. controls (14.00,3.66) and (14.00,2.34) .. (snake_rhst1.north);
  \draw[at wire rev] (snake_rhss2.south) .. controls (15.00,3.66) and (15.00,2.34) .. (snake_rhst2.north);
  \node[rigid node] at (8.7,-0.85) {$\left(\id_{w^\vee} \otimes \ev_w\right)\circ\left(\coev_w \otimes \id_{w^\vee}\right)=\id_{w^\vee}$};
\end{tikzpicture}
\end{center}

\begin{theorem}\label{thm:pivotal}
$\Ddna$ is a strict pivotal monoidal category.
\end{theorem}

\begin{proof}
The monoidal structure (concatenation, empty word, horizontal juxtaposition) is strict by construction. Proposition~\ref{prop:snake} establishes that every object $w$ has a right dual $w\dv$ with the required evaluation and coevaluation morphisms. The same data, read with the opposite convention, give left duals, so $\Ddna$ is rigid. Since $(w\dv)\dv = w$ for every word $w$, the double-dual functor is the identity. This is a (trivial) monoidal natural isomorphism $A \to A^{**}$, so $\Ddna$ is strictly pivotal.
\end{proof}

\subsection{The bending correspondence}

A standard consequence of rigidity is the bijection $\Hom(x, y) \cong \Hom(\eps,\, x\dv \otimes y)$, given by $\widehat{f} := (\id_{x\dv}\otimes f)\circ \coev_x \colon \eps \to x\dv \otimes y$. In the DNA setting, any interaction diagram from $x$ to $y$ is equivalent to a secondary-structure-type diagram on the combined boundary word $x\dv \otimes y$. For example, taking $x=\texttt{ATCGC}$, $y=\texttt{AGCTCG}$, and a morphism $f\colon x\to y$ with both internal arcs and through-strands:
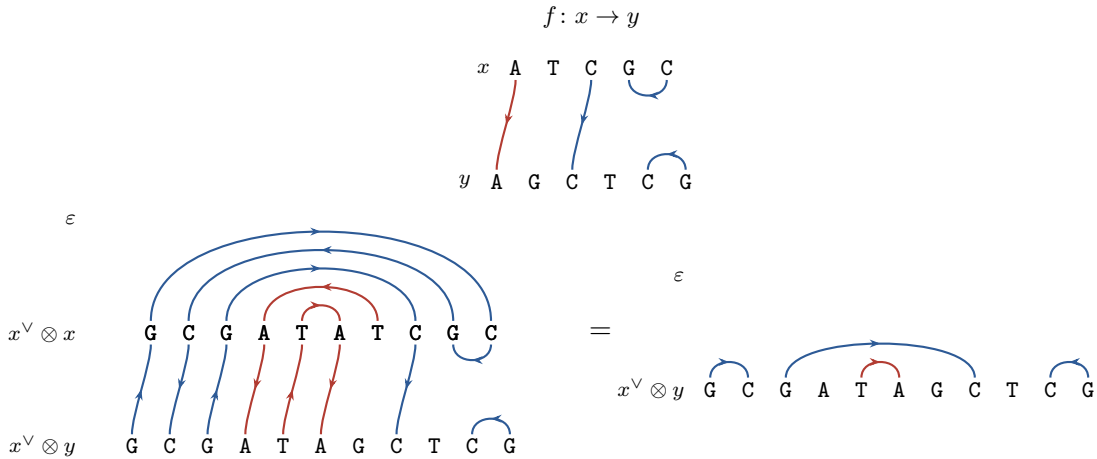
\[
\begin{tikzpicture}[x=0.50cm, y=0.50cm]
  \node[base node] (bend_fs0) at (10.65,10.00) {A};
  \node[base node] (bend_fs1) at (11.65,10.00) {T};
  \node[base node] (bend_fs2) at (12.65,10.00) {C};
  \node[base node] (bend_fs3) at (13.65,10.00) {G};
  \node[base node] (bend_fs4) at (14.65,10.00) {C};
  \node[rigid node, left] at (10.05,10.00) {\footnotesize$x$};
  \node[base node] (bend_ft0) at (10.15,7.00) {A};
  \node[base node] (bend_ft1) at (11.15,7.00) {G};
  \node[base node] (bend_ft2) at (12.15,7.00) {C};
  \node[base node] (bend_ft3) at (13.15,7.00) {T};
  \node[base node] (bend_ft4) at (14.15,7.00) {C};
  \node[base node] (bend_ft5) at (15.15,7.00) {G};
  \node[rigid node, left] at (9.55,7.00) {\footnotesize$y$};
  \draw[cg arc rev] (bend_fs3.south) .. controls +(0,-0.55) and +(0,-0.55) .. (bend_fs4.south);
  \draw[cg arc rev] (bend_ft4.north) .. controls +(0,0.55) and +(0,0.55) .. (bend_ft5.north);
  \draw[at wire] (bend_fs0.south) .. controls (10.65,9.16) and (10.15,7.84) .. (bend_ft0.north);
  \draw[cg wire] (bend_fs2.south) .. controls (12.65,9.16) and (12.15,7.84) .. (bend_ft2.north);
  \node[rigid node] at (12.65,11.35) {\small $f \colon x \to y$};
  \node[rigid node, left] at (-0.90,6.00) {\footnotesize$\eps$};
  \node[rigid node, left] at (-0.90,3.00) {\footnotesize$x^\vee \otimes x$};
  \node[rigid node, left] at (-0.90,0.00) {\footnotesize$x^\vee \otimes y$};
  \node[base node] (m0t0) at (1.00,3.00) {G};
  \node[base node] (m0t1) at (2.00,3.00) {C};
  \node[base node] (m0t2) at (3.00,3.00) {G};
  \node[base node] (m0t3) at (4.00,3.00) {A};
  \node[base node] (m0t4) at (5.00,3.00) {T};
  \node[base node] (m0t5) at (6.00,3.00) {A};
  \node[base node] (m0t6) at (7.00,3.00) {T};
  \node[base node] (m0t7) at (8.00,3.00) {C};
  \node[base node] (m0t8) at (9.00,3.00) {G};
  \node[base node] (m0t9) at (10.00,3.00) {C};
  \draw[cg arc] (m0t0.north) .. controls +(0,3.15) and +(0,3.15) .. (m0t9.north);
  \draw[cg arc rev] (m0t1.north) .. controls +(0,2.50) and +(0,2.50) .. (m0t8.north);
  \draw[cg arc] (m0t2.north) .. controls +(0,1.85) and +(0,1.85) .. (m0t7.north);
  \draw[at arc rev] (m0t3.north) .. controls +(0,1.20) and +(0,1.20) .. (m0t6.north);
  \draw[at arc] (m0t4.north) .. controls +(0,0.55) and +(0,0.55) .. (m0t5.north);
  \node[base node] (m1s0) at (1.00,3.00) {G};
  \node[base node] (m1s1) at (2.00,3.00) {C};
  \node[base node] (m1s2) at (3.00,3.00) {G};
  \node[base node] (m1s3) at (4.00,3.00) {A};
  \node[base node] (m1s4) at (5.00,3.00) {T};
  \node[base node] (m1s5) at (6.00,3.00) {A};
  \node[base node] (m1s6) at (7.00,3.00) {T};
  \node[base node] (m1s7) at (8.00,3.00) {C};
  \node[base node] (m1s8) at (9.00,3.00) {G};
  \node[base node] (m1s9) at (10.00,3.00) {C};
  \node[base node] (m1t0) at (0.50,0.00) {G};
  \node[base node] (m1t1) at (1.50,0.00) {C};
  \node[base node] (m1t2) at (2.50,0.00) {G};
  \node[base node] (m1t3) at (3.50,0.00) {A};
  \node[base node] (m1t4) at (4.50,0.00) {T};
  \node[base node] (m1t5) at (5.50,0.00) {A};
  \node[base node] (m1t6) at (6.50,0.00) {G};
  \node[base node] (m1t7) at (7.50,0.00) {C};
  \node[base node] (m1t8) at (8.50,0.00) {T};
  \node[base node] (m1t9) at (9.50,0.00) {C};
  \node[base node] (m1t10) at (10.50,0.00) {G};
  \draw[cg arc rev] (m1s8.south) .. controls +(0,-0.55) and +(0,-0.55) .. (m1s9.south);
  \draw[cg arc rev] (m1t9.north) .. controls +(0,0.55) and +(0,0.55) .. (m1t10.north);
  \draw[cg wire rev] (m1s0.south) .. controls (1.00,2.16) and (0.50,0.84) .. (m1t0.north);
  \draw[cg wire] (m1s1.south) .. controls (2.00,2.16) and (1.50,0.84) .. (m1t1.north);
  \draw[cg wire rev] (m1s2.south) .. controls (3.00,2.16) and (2.50,0.84) .. (m1t2.north);
  \draw[at wire] (m1s3.south) .. controls (4.00,2.16) and (3.50,0.84) .. (m1t3.north);
  \draw[at wire rev] (m1s4.south) .. controls (5.00,2.16) and (4.50,0.84) .. (m1t4.north);
  \draw[at wire] (m1s5.south) .. controls (6.00,2.16) and (5.50,0.84) .. (m1t5.north);
  \draw[cg wire] (m1s7.south) .. controls (8.00,2.16) and (7.50,0.84) .. (m1t7.north);
  \node[rigid eq] at (12.9,3.0) {$=$};
  \node[rigid node, left] at (15.20,4.50) {\footnotesize$\eps$};
  \node[rigid node, left] at (15.20,1.50) {\footnotesize$x^\vee \otimes y$};
  \node[base node] (bend_rhst0) at (15.80,1.50) {G};
  \node[base node] (bend_rhst1) at (16.80,1.50) {C};
  \node[base node] (bend_rhst2) at (17.80,1.50) {G};
  \node[base node] (bend_rhst3) at (18.80,1.50) {A};
  \node[base node] (bend_rhst4) at (19.80,1.50) {T};
  \node[base node] (bend_rhst5) at (20.80,1.50) {A};
  \node[base node] (bend_rhst6) at (21.80,1.50) {G};
  \node[base node] (bend_rhst7) at (22.80,1.50) {C};
  \node[base node] (bend_rhst8) at (23.80,1.50) {T};
  \node[base node] (bend_rhst9) at (24.80,1.50) {C};
  \node[base node] (bend_rhst10) at (25.80,1.50) {G};
  \draw[cg arc rev] (bend_rhst9.north) .. controls +(0,0.55) and +(0,0.55) .. (bend_rhst10.north);
  \draw[cg arc] (bend_rhst2.north) .. controls +(0,1.20) and +(0,1.20) .. (bend_rhst7.north);
  \draw[at arc] (bend_rhst4.north) .. controls +(0,0.55) and +(0,0.55) .. (bend_rhst5.north);
  \draw[cg arc] (bend_rhst0.north) .. controls +(0,0.55) and +(0,0.55) .. (bend_rhst1.north);
\end{tikzpicture}
\]

\subsection{The straightened picture and zip-and-transfer}\label{subsec:strand}

The bending correspondence gives the operational punchline of the construction: composition in $\Ddna$ corresponds to an interface-level zip-and-transfer mechanism that can be implemented in DNA molecular computation under standard design constraints.\footnote{This correspondence is operational when each evaluation step is thermodynamically downhill and kinetically accessible. In strand-displacement systems, toeholds and branch-migration domains are designed so intended duplex products are favored over competitors~\cite{SantaLucia1998,Zhang2009,Qian2011}.} Given morphisms $f \colon x \to y$ and $g \colon y \to z$, let
\[
\widehat{f} \colon \eps \to x\dv \otimes y, \qquad \widehat{g} \colon \eps \to y\dv \otimes z
\]
be their straightened forms. These are secondary-structure states---morphisms from the empty word. The composite $g \circ f \colon x \to z$ is computed entirely at the secondary-structure level in three steps:

\begin{enumerate}
  \item \textbf{Juxtapose.} Place $\widehat{f}$ and $\widehat{g}$ side by side (tensor), obtaining a state
  \[
  \widehat{f} \otimes \widehat{g} \colon \eps \to x\dv \otimes y \otimes y\dv \otimes z.
  \]

  \item \textbf{Contract.} The adjacent segments $y$ and $y\dv$ are Watson--Crick complements. Contract them by applying the canonical evaluation $\ev_y \colon y \otimes y\dv \to \eps$. This pairs each base in $y$ with its complement in $y\dv$.

  \item \textbf{Simplify.} Straighten the surviving connectivity and erase closed components that arise entirely within the contracted region (normalized loop convention).
\end{enumerate}

The result is $\widehat{g \circ f} \colon \eps \to x\dv \otimes z$.

As an elementary discrete thermodynamic model, assign one unit to each base-pair bond and treat reachable structures with maximal bond count as stable; bond transfer itself is a rewiring step that preserves bond count, while the full zip-and-transfer trajectory is favorable when additional interface bonds form and the final state has a larger total bond count than the isolated initial states.

What happens combinatorially in step (2) is a \emph{zip-and-transfer} operation. The complementary segments $y$ and $y\dv$ zip together---each base in $y$ pairs with its complement in $y\dv$. But some of these positions already participate in arcs from $\widehat{f}$ or $\widehat{g}$. When a position in $y$ that is paired (via $\widehat{f}$) to some position in $x\dv$ zips together with its complement in $y\dv$ that is paired (via $\widehat{g}$) to some position in $z$, the base-pair connectivity \emph{transfers through} the interface: the $x\dv$ position ends up paired to the $z$ position. Arcs that do not bridge the interface close into loops and are erased.

For a concrete generic instance with nontrivial structure on all three boundaries, take
\[
x=\texttt{CGCGT},\qquad y=\texttt{CGAAGG},\qquad z=\texttt{CGCTATC},
\]
and morphisms $f\colon x\to y$, $g\colon y\to z$ chosen so that the straightened composition has two transfer paths and one surviving u-turn:
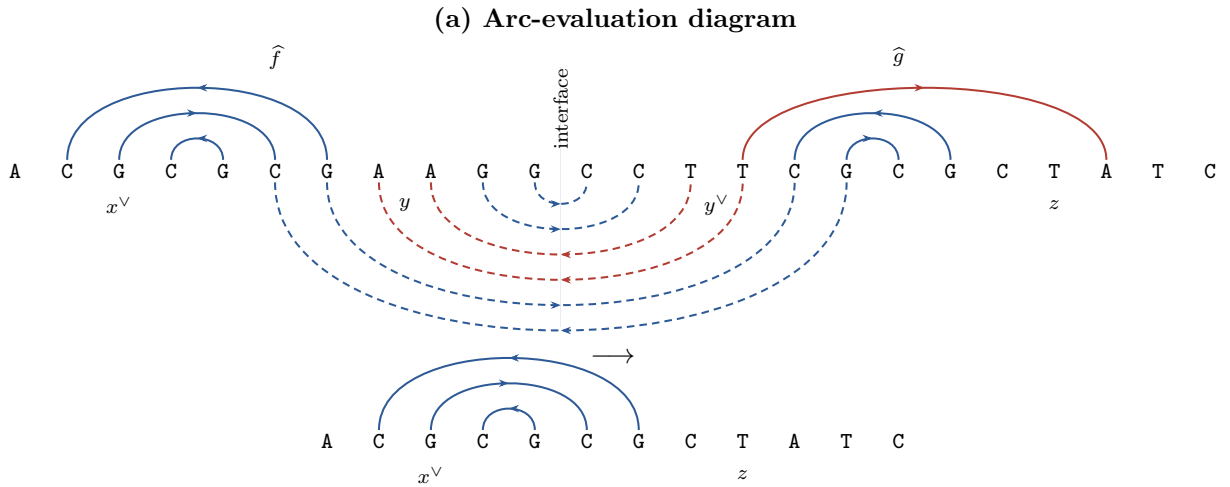
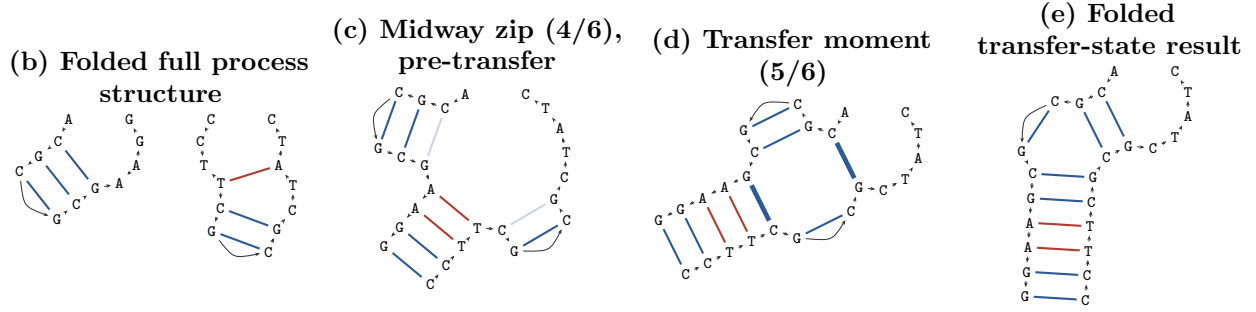
\begin{center}
\begin{minipage}{\textwidth}
  \centering
  \textbf{(a) Arc-evaluation diagram}
  \par\smallskip
  \resizebox{0.98\textwidth}{!}{
\begin{tikzpicture}[x=0.72cm, y=0.72cm]
  \begin{scope}[on background layer]
    \fill[dnaInterface] (11.5,-3.0) rectangle (11.5,2.0);
  \end{scope}
  \node[base node] (b1) at (1,0) {A};
  \node[base node] (b2) at (2,0) {C};
  \node[base node] (b3) at (3,0) {G};
  \node[base node] (b4) at (4,0) {C};
  \node[base node] (b5) at (5,0) {G};
  \node[base node] (b6) at (6,0) {C};
  \node[base node] (b7) at (7,0) {G};
  \node[base node] (b8) at (8,0) {A};
  \node[base node] (b9) at (9,0) {A};
  \node[base node] (b10) at (10,0) {G};
  \node[base node] (b11) at (11,0) {G};
  \node[base node] (b12) at (12,0) {C};
  \node[base node] (b13) at (13,0) {C};
  \node[base node] (b14) at (14,0) {T};
  \node[base node] (b15) at (15,0) {T};
  \node[base node] (b16) at (16,0) {C};
  \node[base node] (b17) at (17,0) {G};
  \node[base node] (b18) at (18,0) {C};
  \node[base node] (b19) at (19,0) {G};
  \node[base node] (b20) at (20,0) {C};
  \node[base node] (b21) at (21,0) {T};
  \node[base node] (b22) at (22,0) {A};
  \node[base node] (b23) at (23,0) {T};
  \node[base node] (b24) at (24,0) {C};
  \node[rigid node] at (21.0,-0.65) {\footnotesize$z$};
  \node[rigid node] at (14.5,-0.65) {\footnotesize$y^\vee$};
  \node[rigid node] at (8.5,-0.65) {\footnotesize$y$};
  \node[rigid node] at (3.0,-0.65) {\footnotesize$x^\vee$};
  \draw[cg arc rev] (b2.north) .. controls +(0,1.85) and +(0,1.85) .. (b7.north);
  \draw[cg arc] (b3.north) .. controls +(0,1.20) and +(0,1.20) .. (b6.north);
  \draw[cg arc rev] (b4.north) .. controls +(0,0.55) and +(0,0.55) .. (b5.north);
  \draw[at arc] (b15.north) .. controls +(0,1.85) and +(0,1.85) .. (b22.north);
  \draw[cg arc rev] (b16.north) .. controls +(0,1.20) and +(0,1.20) .. (b19.north);
  \draw[cg arc] (b17.north) .. controls +(0,0.55) and +(0,0.55) .. (b18.north);
  \draw[cg arc rev, densely dashed] (b6.south) .. controls +(0,-3.80) and +(0,-3.80) .. (b17.south);
  \draw[cg arc, densely dashed] (b7.south) .. controls +(0,-3.15) and +(0,-3.15) .. (b16.south);
  \draw[at arc rev, densely dashed] (b8.south) .. controls +(0,-2.50) and +(0,-2.50) .. (b15.south);
  \draw[at arc rev, densely dashed] (b9.south) .. controls +(0,-1.85) and +(0,-1.85) .. (b14.south);
  \draw[cg arc, densely dashed] (b10.south) .. controls +(0,-1.20) and +(0,-1.20) .. (b13.south);
  \draw[cg arc, densely dashed] (b11.south) .. controls +(0,-0.55) and +(0,-0.55) .. (b12.south);
  \node[rigid node] at (6.0,2.2) {\footnotesize$\widehat{f}$};
  \node[rigid node] at (18.0,2.2) {\footnotesize$\widehat{g}$};
  \node[rigid node, rotate=90] at (11.5,1.2) {\footnotesize interface};
  \node[rigid eq] at (12.5,-3.6) {$\longrightarrow$};
  \node[base node] (r1) at (7.0,-5.2) {A};
  \node[base node] (r2) at (8.0,-5.2) {C};
  \node[base node] (r3) at (9.0,-5.2) {G};
  \node[base node] (r4) at (10.0,-5.2) {C};
  \node[base node] (r5) at (11.0,-5.2) {G};
  \node[base node] (r6) at (12.0,-5.2) {C};
  \node[base node] (r7) at (13.0,-5.2) {G};
  \node[base node] (r8) at (14.0,-5.2) {C};
  \node[base node] (r9) at (15.0,-5.2) {T};
  \node[base node] (r10) at (16.0,-5.2) {A};
  \node[base node] (r11) at (17.0,-5.2) {T};
  \node[base node] (r12) at (18.0,-5.2) {C};
  \draw[cg arc rev] (r2.north) .. controls +(0,1.85) and +(0,1.85) .. (r7.north);
  \draw[cg arc] (r3.north) .. controls +(0,1.20) and +(0,1.20) .. (r6.north);
  \draw[cg arc rev] (r4.north) .. controls +(0,0.55) and +(0,0.55) .. (r5.north);
  \node[rigid node] at (9.0,-5.85) {\footnotesize$x^\vee$};
  \node[rigid node] at (15.0,-5.85) {\footnotesize$z$};
\end{tikzpicture}}
\end{minipage}

\vspace{0.9em}

\begin{minipage}{0.24\textwidth}
  \centering
  \textbf{(b) Folded full process structure}
  \par\smallskip
  \scalebox{1.0}{
\begin{tikzpicture}[x=0.5cm, y=0.5cm]
  \node[base node, font=\tiny\ttfamily] (b0) at (1.305,0.749) {A};
  \node[base node, font=\tiny\ttfamily] (b1) at (1.021,0.164) {C};
  \node[base node, font=\tiny\ttfamily] (b2) at (0.511,-0.238) {G};
  \node[base node, font=\tiny\ttfamily] (b3) at (0.000,-0.641) {C};
  \node[base node, font=\tiny\ttfamily] (b4) at (0.991,-1.897) {G};
  \node[base node, font=\tiny\ttfamily] (b5) at (1.502,-1.495) {C};
  \node[base node, font=\tiny\ttfamily] (b6) at (2.012,-1.092) {G};
  \node[base node, font=\tiny\ttfamily] (b7) at (2.647,-0.953) {A};
  \node[base node, font=\tiny\ttfamily] (b8) at (3.091,-0.478) {A};
  \node[base node, font=\tiny\ttfamily] (b9) at (3.189,0.164) {G};
  \node[base node, font=\tiny\ttfamily] (b10) at (2.905,0.749) {G};
  \draw[-{Stealth[length=1.6pt,width=1.2pt]}, dnaWire, thin] (b0) -- (b1);
  \draw[-{Stealth[length=1.6pt,width=1.2pt]}, dnaWire, thin] (b1) -- (b2);
  \draw[-{Stealth[length=1.6pt,width=1.2pt]}, dnaWire, thin] (b2) -- (b3);
  \draw[-{Stealth[length=1.6pt,width=1.2pt]}, dnaWire, thin] (b3) .. controls (0.113,-1.571) and (0.113,-1.571) .. (b4);
  \draw[-{Stealth[length=1.6pt,width=1.2pt]}, dnaWire, thin] (b4) -- (b5);
  \draw[-{Stealth[length=1.6pt,width=1.2pt]}, dnaWire, thin] (b5) -- (b6);
  \draw[-{Stealth[length=1.6pt,width=1.2pt]}, dnaWire, thin] (b6) -- (b7);
  \draw[-{Stealth[length=1.6pt,width=1.2pt]}, dnaWire, thin] (b7) -- (b8);
  \draw[-{Stealth[length=1.6pt,width=1.2pt]}, dnaWire, thin] (b8) -- (b9);
  \draw[-{Stealth[length=1.6pt,width=1.2pt]}, dnaWire, thin] (b9) -- (b10);
  \draw[dnaCG, thick] (b1) -- (b6);
  \draw[dnaCG, thick] (b2) -- (b5);
  \draw[dnaCG, thick] (b3) -- (b4);
  \node[base node, font=\tiny\ttfamily] (b11) at (5.073,0.749) {C};
  \node[base node, font=\tiny\ttfamily] (b12) at (4.789,0.164) {C};
  \node[base node, font=\tiny\ttfamily] (b13) at (4.887,-0.478) {T};
  \node[base node, font=\tiny\ttfamily] (b14) at (5.331,-0.953) {T};
  \node[base node, font=\tiny\ttfamily] (b15) at (5.314,-1.602) {C};
  \node[base node, font=\tiny\ttfamily] (b16) at (5.088,-2.212) {G};
  \node[base node, font=\tiny\ttfamily] (b17) at (6.588,-2.768) {C};
  \node[base node, font=\tiny\ttfamily] (b18) at (6.814,-2.158) {G};
  \node[base node, font=\tiny\ttfamily] (b19) at (7.224,-1.654) {C};
  \node[base node, font=\tiny\ttfamily] (b20) at (7.242,-1.004) {T};
  \node[base node, font=\tiny\ttfamily] (b21) at (6.859,-0.478) {A};
  \node[base node, font=\tiny\ttfamily] (b22) at (6.957,0.164) {T};
  \node[base node, font=\tiny\ttfamily] (b23) at (6.673,0.749) {C};
  \draw[-{Stealth[length=1.6pt,width=1.2pt]}, dnaWire, thin] (b11) -- (b12);
  \draw[-{Stealth[length=1.6pt,width=1.2pt]}, dnaWire, thin] (b12) -- (b13);
  \draw[-{Stealth[length=1.6pt,width=1.2pt]}, dnaWire, thin] (b13) -- (b14);
  \draw[-{Stealth[length=1.6pt,width=1.2pt]}, dnaWire, thin] (b14) -- (b15);
  \draw[-{Stealth[length=1.6pt,width=1.2pt]}, dnaWire, thin] (b15) -- (b16);
  \draw[-{Stealth[length=1.6pt,width=1.2pt]}, dnaWire, thin] (b16) .. controls (5.669,-2.947) and (5.669,-2.947) .. (b17);
  \draw[-{Stealth[length=1.6pt,width=1.2pt]}, dnaWire, thin] (b17) -- (b18);
  \draw[-{Stealth[length=1.6pt,width=1.2pt]}, dnaWire, thin] (b18) -- (b19);
  \draw[-{Stealth[length=1.6pt,width=1.2pt]}, dnaWire, thin] (b19) -- (b20);
  \draw[-{Stealth[length=1.6pt,width=1.2pt]}, dnaWire, thin] (b20) -- (b21);
  \draw[-{Stealth[length=1.6pt,width=1.2pt]}, dnaWire, thin] (b21) -- (b22);
  \draw[-{Stealth[length=1.6pt,width=1.2pt]}, dnaWire, thin] (b22) -- (b23);
  \draw[dnaAT, thick] (b14) -- (b21);
  \draw[dnaCG, thick] (b15) -- (b18);
  \draw[dnaCG, thick] (b16) -- (b17);
\end{tikzpicture}}
\end{minipage}\hfill
\begin{minipage}{0.24\textwidth}
  \centering
  \textbf{(c) Midway zip (4/6), pre-transfer}
  \par\smallskip
  \scalebox{1.0}{
\begin{tikzpicture}[x=0.5cm, y=0.5cm]
  \node[base node, font=\tiny\ttfamily] (b0) at (-0.800,1.705) {A};
  \node[base node, font=\tiny\ttfamily] (b1) at (-1.332,1.332) {C};
  \node[base node, font=\tiny\ttfamily] (b2) at (-1.943,1.553) {G};
  \node[base node, font=\tiny\ttfamily] (b3) at (-2.555,1.774) {C};
  \node[base node, font=\tiny\ttfamily] (b4) at (-3.098,0.269) {G};
  \node[base node, font=\tiny\ttfamily] (b5) at (-2.487,0.048) {C};
  \node[base node, font=\tiny\ttfamily] (b6) at (-1.876,-0.173) {G};
  \node[base node, font=\tiny\ttfamily] (b7) at (-1.705,-0.800) {A};
  \node[base node, font=\tiny\ttfamily] (b8) at (-2.121,-1.300) {A};
  \node[base node, font=\tiny\ttfamily] (b9) at (-2.537,-1.800) {G};
  \node[base node, font=\tiny\ttfamily] (b10) at (-2.952,-2.299) {G};
  \node[base node, font=\tiny\ttfamily] (b11) at (-1.722,-3.322) {C};
  \node[base node, font=\tiny\ttfamily] (b12) at (-1.306,-2.823) {C};
  \node[base node, font=\tiny\ttfamily] (b13) at (-0.891,-2.323) {T};
  \node[base node, font=\tiny\ttfamily] (b14) at (-0.475,-1.823) {T};
  \node[base node, font=\tiny\ttfamily] (b15) at (0.173,-1.876) {C};
  \node[base node, font=\tiny\ttfamily] (b16) at (0.502,-2.437) {G};
  \node[base node, font=\tiny\ttfamily] (b17) at (1.882,-1.627) {C};
  \node[base node, font=\tiny\ttfamily] (b18) at (1.553,-1.066) {G};
  \node[base node, font=\tiny\ttfamily] (b19) at (1.823,-0.475) {C};
  \node[base node, font=\tiny\ttfamily] (b20) at (1.876,0.173) {T};
  \node[base node, font=\tiny\ttfamily] (b21) at (1.705,0.800) {A};
  \node[base node, font=\tiny\ttfamily] (b22) at (1.332,1.332) {T};
  \node[base node, font=\tiny\ttfamily] (b23) at (0.800,1.705) {C};
  \draw[-{Stealth[length=1.6pt,width=1.2pt]}, dnaWire, thin] (b0) -- (b1);
  \draw[-{Stealth[length=1.6pt,width=1.2pt]}, dnaWire, thin] (b1) -- (b2);
  \draw[-{Stealth[length=1.6pt,width=1.2pt]}, dnaWire, thin] (b2) -- (b3);
  \draw[-{Stealth[length=1.6pt,width=1.2pt]}, dnaWire, thin] (b3) .. controls (-3.285,1.187) and (-3.285,1.187) .. (b4);
  \draw[-{Stealth[length=1.6pt,width=1.2pt]}, dnaWire, thin] (b4) -- (b5);
  \draw[-{Stealth[length=1.6pt,width=1.2pt]}, dnaWire, thin] (b5) -- (b6);
  \draw[-{Stealth[length=1.6pt,width=1.2pt]}, dnaWire, thin] (b6) -- (b7);
  \draw[-{Stealth[length=1.6pt,width=1.2pt]}, dnaWire, thin] (b7) -- (b8);
  \draw[-{Stealth[length=1.6pt,width=1.2pt]}, dnaWire, thin] (b8) -- (b9);
  \draw[-{Stealth[length=1.6pt,width=1.2pt]}, dnaWire, thin] (b9) -- (b10);
  \draw[-{Stealth[length=1.6pt,width=1.2pt]}, dnaWire, thin] (b11) -- (b12);
  \draw[-{Stealth[length=1.6pt,width=1.2pt]}, dnaWire, thin] (b12) -- (b13);
  \draw[-{Stealth[length=1.6pt,width=1.2pt]}, dnaWire, thin] (b13) -- (b14);
  \draw[-{Stealth[length=1.6pt,width=1.2pt]}, dnaWire, thin] (b14) -- (b15);
  \draw[-{Stealth[length=1.6pt,width=1.2pt]}, dnaWire, thin] (b15) -- (b16);
  \draw[-{Stealth[length=1.6pt,width=1.2pt]}, dnaWire, thin] (b16) .. controls (1.438,-2.452) and (1.438,-2.452) .. (b17);
  \draw[-{Stealth[length=1.6pt,width=1.2pt]}, dnaWire, thin] (b17) -- (b18);
  \draw[-{Stealth[length=1.6pt,width=1.2pt]}, dnaWire, thin] (b18) -- (b19);
  \draw[-{Stealth[length=1.6pt,width=1.2pt]}, dnaWire, thin] (b19) -- (b20);
  \draw[-{Stealth[length=1.6pt,width=1.2pt]}, dnaWire, thin] (b20) -- (b21);
  \draw[-{Stealth[length=1.6pt,width=1.2pt]}, dnaWire, thin] (b21) -- (b22);
  \draw[-{Stealth[length=1.6pt,width=1.2pt]}, dnaWire, thin] (b22) -- (b23);
  \draw[dnaCG, thick, opacity=0.25] (b1) -- (b6);
  \draw[dnaCG, thick] (b2) -- (b5);
  \draw[dnaCG, thick] (b3) -- (b4);
  \draw[dnaAT, thick] (b7) -- (b14);
  \draw[dnaAT, thick] (b8) -- (b13);
  \draw[dnaCG, thick] (b9) -- (b12);
  \draw[dnaCG, thick] (b10) -- (b11);
  \draw[dnaCG, thick, opacity=0.25] (b15) -- (b18);
  \draw[dnaCG, thick] (b16) -- (b17);
\end{tikzpicture}}
\end{minipage}\hfill
\begin{minipage}{0.24\textwidth}
  \centering
  \textbf{(d) Transfer moment (5/6)}
  \par\smallskip
  \scalebox{1.0}{
\begin{tikzpicture}[x=0.5cm, y=0.5cm]
  \node[base node, font=\tiny\ttfamily] (b0) at (-0.800,0.878) {A};
  \node[base node, font=\tiny\ttfamily] (b1) at (-1.142,0.325) {C};
  \node[base node, font=\tiny\ttfamily] (b2) at (-1.758,0.532) {G};
  \node[base node, font=\tiny\ttfamily] (b3) at (-2.048,1.113) {C};
  \node[base node, font=\tiny\ttfamily] (b4) at (-3.480,0.400) {G};
  \node[base node, font=\tiny\ttfamily] (b5) at (-3.191,-0.182) {C};
  \node[base node, font=\tiny\ttfamily] (b6) at (-3.397,-0.798) {G};
  \node[base node, font=\tiny\ttfamily] (b7) at (-3.979,-1.087) {A};
  \node[base node, font=\tiny\ttfamily] (b8) at (-4.561,-1.377) {A};
  \node[base node, font=\tiny\ttfamily] (b9) at (-5.143,-1.667) {G};
  \node[base node, font=\tiny\ttfamily] (b10) at (-5.725,-1.957) {G};
  \node[base node, font=\tiny\ttfamily] (b11) at (-5.012,-3.389) {C};
  \node[base node, font=\tiny\ttfamily] (b12) at (-4.430,-3.099) {C};
  \node[base node, font=\tiny\ttfamily] (b13) at (-3.848,-2.809) {T};
  \node[base node, font=\tiny\ttfamily] (b14) at (-3.266,-2.520) {T};
  \node[base node, font=\tiny\ttfamily] (b15) at (-2.684,-2.230) {C};
  \node[base node, font=\tiny\ttfamily] (b16) at (-2.068,-2.437) {G};
  \node[base node, font=\tiny\ttfamily] (b17) at (-0.636,-1.724) {C};
  \node[base node, font=\tiny\ttfamily] (b18) at (-0.429,-1.107) {G};
  \node[base node, font=\tiny\ttfamily] (b19) at (0.218,-1.167) {C};
  \node[base node, font=\tiny\ttfamily] (b20) at (0.800,-0.878) {T};
  \node[base node, font=\tiny\ttfamily] (b21) at (1.142,-0.325) {A};
  \node[base node, font=\tiny\ttfamily] (b22) at (1.142,0.325) {T};
  \node[base node, font=\tiny\ttfamily] (b23) at (0.800,0.878) {C};
  \draw[-{Stealth[length=1.6pt,width=1.2pt]}, dnaWire, thin] (b0) -- (b1);
  \draw[-{Stealth[length=1.6pt,width=1.2pt]}, dnaWire, thin] (b1) -- (b2);
  \draw[-{Stealth[length=1.6pt,width=1.2pt]}, dnaWire, thin] (b2) -- (b3);
  \draw[-{Stealth[length=1.6pt,width=1.2pt]}, dnaWire, thin] (b3) .. controls (-2.982,1.193) and (-2.982,1.193) .. (b4);
  \draw[-{Stealth[length=1.6pt,width=1.2pt]}, dnaWire, thin] (b4) -- (b5);
  \draw[-{Stealth[length=1.6pt,width=1.2pt]}, dnaWire, thin] (b5) -- (b6);
  \draw[-{Stealth[length=1.6pt,width=1.2pt]}, dnaWire, thin] (b6) -- (b7);
  \draw[-{Stealth[length=1.6pt,width=1.2pt]}, dnaWire, thin] (b7) -- (b8);
  \draw[-{Stealth[length=1.6pt,width=1.2pt]}, dnaWire, thin] (b8) -- (b9);
  \draw[-{Stealth[length=1.6pt,width=1.2pt]}, dnaWire, thin] (b9) -- (b10);
  \draw[-{Stealth[length=1.6pt,width=1.2pt]}, dnaWire, thin] (b11) -- (b12);
  \draw[-{Stealth[length=1.6pt,width=1.2pt]}, dnaWire, thin] (b12) -- (b13);
  \draw[-{Stealth[length=1.6pt,width=1.2pt]}, dnaWire, thin] (b13) -- (b14);
  \draw[-{Stealth[length=1.6pt,width=1.2pt]}, dnaWire, thin] (b14) -- (b15);
  \draw[-{Stealth[length=1.6pt,width=1.2pt]}, dnaWire, thin] (b15) -- (b16);
  \draw[-{Stealth[length=1.6pt,width=1.2pt]}, dnaWire, thin] (b16) .. controls (-1.135,-2.517) and (-1.135,-2.517) .. (b17);
  \draw[-{Stealth[length=1.6pt,width=1.2pt]}, dnaWire, thin] (b17) -- (b18);
  \draw[-{Stealth[length=1.6pt,width=1.2pt]}, dnaWire, thin] (b18) -- (b19);
  \draw[-{Stealth[length=1.6pt,width=1.2pt]}, dnaWire, thin] (b19) -- (b20);
  \draw[-{Stealth[length=1.6pt,width=1.2pt]}, dnaWire, thin] (b20) -- (b21);
  \draw[-{Stealth[length=1.6pt,width=1.2pt]}, dnaWire, thin] (b21) -- (b22);
  \draw[-{Stealth[length=1.6pt,width=1.2pt]}, dnaWire, thin] (b22) -- (b23);
  \draw[dnaCG, thick, line width=1.8pt] (b1) -- (b18);
  \draw[dnaCG, thick] (b2) -- (b5);
  \draw[dnaCG, thick] (b3) -- (b4);
  \draw[dnaCG, thick, line width=1.8pt] (b6) -- (b15);
  \draw[dnaAT, thick] (b7) -- (b14);
  \draw[dnaAT, thick] (b8) -- (b13);
  \draw[dnaCG, thick] (b9) -- (b12);
  \draw[dnaCG, thick] (b10) -- (b11);
  \draw[dnaCG, thick] (b16) -- (b17);
\end{tikzpicture}}
\end{minipage}\hfill
\begin{minipage}{0.24\textwidth}
  \centering
  \textbf{(e) Folded transfer-state result}
  \par\smallskip
  \scalebox{1.0}{
\begin{tikzpicture}[x=0.5cm, y=0.5cm]
  \node[base node, font=\tiny\ttfamily] (b0) at (-0.800,0.878) {A};
  \node[base node, font=\tiny\ttfamily] (b1) at (-1.142,0.325) {C};
  \node[base node, font=\tiny\ttfamily] (b2) at (-1.724,0.035) {G};
  \node[base node, font=\tiny\ttfamily] (b3) at (-2.373,0.075) {C};
  \node[base node, font=\tiny\ttfamily] (b4) at (-3.257,-1.258) {G};
  \node[base node, font=\tiny\ttfamily] (b5) at (-2.967,-1.840) {C};
  \node[base node, font=\tiny\ttfamily] (b6) at (-3.007,-2.489) {G};
  \node[base node, font=\tiny\ttfamily] (b7) at (-3.047,-3.138) {A};
  \node[base node, font=\tiny\ttfamily] (b8) at (-3.087,-3.787) {A};
  \node[base node, font=\tiny\ttfamily] (b9) at (-3.127,-4.435) {G};
  \node[base node, font=\tiny\ttfamily] (b10) at (-3.167,-5.084) {G};
  \node[base node, font=\tiny\ttfamily] (b11) at (-1.570,-5.183) {C};
  \node[base node, font=\tiny\ttfamily] (b12) at (-1.530,-4.534) {C};
  \node[base node, font=\tiny\ttfamily] (b13) at (-1.490,-3.885) {T};
  \node[base node, font=\tiny\ttfamily] (b14) at (-1.450,-3.236) {T};
  \node[base node, font=\tiny\ttfamily] (b15) at (-1.410,-2.588) {C};
  \node[base node, font=\tiny\ttfamily] (b16) at (-1.370,-1.939) {G};
  \node[base node, font=\tiny\ttfamily] (b17) at (-1.011,-1.397) {C};
  \node[base node, font=\tiny\ttfamily] (b18) at (-0.429,-1.107) {G};
  \node[base node, font=\tiny\ttfamily] (b19) at (0.218,-1.167) {C};
  \node[base node, font=\tiny\ttfamily] (b20) at (0.800,-0.878) {T};
  \node[base node, font=\tiny\ttfamily] (b21) at (1.142,-0.325) {A};
  \node[base node, font=\tiny\ttfamily] (b22) at (1.142,0.325) {T};
  \node[base node, font=\tiny\ttfamily] (b23) at (0.800,0.878) {C};
  \draw[-{Stealth[length=1.6pt,width=1.2pt]}, dnaWire, thin] (b0) -- (b1);
  \draw[-{Stealth[length=1.6pt,width=1.2pt]}, dnaWire, thin] (b1) -- (b2);
  \draw[-{Stealth[length=1.6pt,width=1.2pt]}, dnaWire, thin] (b2) -- (b3);
  \draw[-{Stealth[length=1.6pt,width=1.2pt]}, dnaWire, thin] (b3) .. controls (-3.221,-0.322) and (-3.221,-0.322) .. (b4);
  \draw[-{Stealth[length=1.6pt,width=1.2pt]}, dnaWire, thin] (b4) -- (b5);
  \draw[-{Stealth[length=1.6pt,width=1.2pt]}, dnaWire, thin] (b5) -- (b6);
  \draw[-{Stealth[length=1.6pt,width=1.2pt]}, dnaWire, thin] (b6) -- (b7);
  \draw[-{Stealth[length=1.6pt,width=1.2pt]}, dnaWire, thin] (b7) -- (b8);
  \draw[-{Stealth[length=1.6pt,width=1.2pt]}, dnaWire, thin] (b8) -- (b9);
  \draw[-{Stealth[length=1.6pt,width=1.2pt]}, dnaWire, thin] (b9) -- (b10);
  \draw[-{Stealth[length=1.6pt,width=1.2pt]}, dnaWire, thin] (b11) -- (b12);
  \draw[-{Stealth[length=1.6pt,width=1.2pt]}, dnaWire, thin] (b12) -- (b13);
  \draw[-{Stealth[length=1.6pt,width=1.2pt]}, dnaWire, thin] (b13) -- (b14);
  \draw[-{Stealth[length=1.6pt,width=1.2pt]}, dnaWire, thin] (b14) -- (b15);
  \draw[-{Stealth[length=1.6pt,width=1.2pt]}, dnaWire, thin] (b15) -- (b16);
  \draw[-{Stealth[length=1.6pt,width=1.2pt]}, dnaWire, thin] (b16) -- (b17);
  \draw[-{Stealth[length=1.6pt,width=1.2pt]}, dnaWire, thin] (b17) -- (b18);
  \draw[-{Stealth[length=1.6pt,width=1.2pt]}, dnaWire, thin] (b18) -- (b19);
  \draw[-{Stealth[length=1.6pt,width=1.2pt]}, dnaWire, thin] (b19) -- (b20);
  \draw[-{Stealth[length=1.6pt,width=1.2pt]}, dnaWire, thin] (b20) -- (b21);
  \draw[-{Stealth[length=1.6pt,width=1.2pt]}, dnaWire, thin] (b21) -- (b22);
  \draw[-{Stealth[length=1.6pt,width=1.2pt]}, dnaWire, thin] (b22) -- (b23);
  \draw[dnaCG, thick] (b1) -- (b18);
  \draw[dnaCG, thick] (b2) -- (b17);
  \draw[dnaCG, thick] (b3) -- (b4);
  \draw[dnaCG, thick] (b5) -- (b16);
  \draw[dnaCG, thick] (b6) -- (b15);
  \draw[dnaAT, thick] (b7) -- (b14);
  \draw[dnaAT, thick] (b8) -- (b13);
  \draw[dnaCG, thick] (b9) -- (b12);
  \draw[dnaCG, thick] (b10) -- (b11);
\end{tikzpicture}}
\end{minipage}
\end{center}

In panels (c)$\to$(d), the highlighted bonds are the same pairings viewed before and after transfer: the bold interface bonds in (c) are rewired into the bold $x^\vee$--$z$ bonds in (d).

Strand displacement~\cite{Yurke2000, Zhang2009, Soloveichik2010} is the best-studied kinetically controlled instance of this mechanism: the interface has the specific structure of a toehold-initiated branch migration, and kinetic control ensures directionality. More broadly, branch migration---the stepwise transfer of base-pair partners along a homologous duplex~\cite{Radding1982}---is the biological process that directly instantiates the transfer operation of our categorical composition. A fully worked specialization with explicit morphisms and diagrams is given in Appendix~\ref{app:tmsd} (Example~\ref{ex:tmsd}); our claim is therefore a mechanistic correspondence at the level of connectivity and design constraints, not a full quantitative model of reaction kinetics or thermodynamic landscapes.

\section{From grammar to DNA}\label{sec:grammar}

The rigid structure of $\Ddna$ matches the categorical backbone of compositional models of natural language.

\subsection{Pregroup grammars and compositional semantics}

A \emph{pregroup}~\cite{Lambek1999} is a partially ordered monoid in which every element $a$ has adjoints $a^l$, $a^r$ satisfying $a^l a \leq 1 \leq a a^l$ and $a a^r \leq 1 \leq a^r a$. A pregroup grammar assigns to each word one or more types from a free pregroup over basic types such as $n$ (noun) and $s$ (sentence); a string is grammatical if the product of its types reduces to $s$. Categorically, a pregroup is a rigid monoidal category that happens to be a poset, with contractions as evaluations and expansions as coevaluations.

Coecke, Sadrzadeh, and Clark~\cite{Coecke2010} observed that this shared rigid structure can be exploited compositionally: a strong monoidal functor $F \colon \mathcal{G} \to \mathcal{S}$ from a grammatical category $\mathcal{G}$ to a semantic target $\mathcal{S}$ transports grammatical reductions into semantic operations. In the original DisCoCat model, $\mathcal{S} = \mathbf{FdVect}$. We propose $\Ddna$ as an alternative target.

\subsection{The functor \texorpdfstring{$F \colon \mathcal{G} \to \Ddna$}{F: G -> D-DNA}}

A strong monoidal functor $F \colon \mathcal{G} \to \Ddna$ is determined by the following data:

\begin{itemize}
  \item \textbf{On basic types.} Each basic grammatical type is assigned a DNA sequence. For instance, one might set $F(n) = N$ and $F(s) = S$ for chosen words $N, S \in \Sig^*$. The choice of sequences is part of the model design and is not determined by the category.

  \item \textbf{On duals.} Since $F$ must respect duality, we have $F(n^r) = F(n)\dv = N\dv$. This is automatic from the monoidal functor requirement.

  \item \textbf{On cups and caps.} The grammatical contraction $n \otimes n^r \leq 1$ maps to the DNA evaluation $\ev_N \colon N \otimes N\dv \to \eps$, which is the canonical Watson--Crick duplex formation. This is the key structural match: \emph{grammatical reduction is base pairing}.

  \item \textbf{On lexical items.} A word of grammatical type $t$ is a state $I \to t$ in the grammatical category (or a lexical extension of it). Under $F$, this becomes a morphism $\eps \to F(t)$---a secondary structure on the DNA sequence $F(t)$. Different words of the same grammatical type correspond to different secondary structures on the same sequence.
\end{itemize}

\begin{example}
Consider the sentence ``Cats chase mice'' with the type assignment
\[
\text{Cats} : n, \qquad \text{chase} : n^r \otimes s \otimes n^l, \qquad \text{mice} : n.
\]
The grammatical reduction contracts $n \otimes n^r$ on the left and $n^l \otimes n$ on the right, yielding the sentence type $s$. Under $F$, this becomes:
\[
\eps \xrightarrow{\; F(\text{Cats}) \otimes F(\text{chase}) \otimes F(\text{mice}) \;}
N \otimes N\dv \otimes S \otimes N\dv \otimes N
\xrightarrow{\; \ev_N \otimes \id_S \otimes \ev_N \;}
S.
\]
The intermediate step---the application of the evaluation maps---performs Watson--Crick pairing between the adjacent complementary segments $N$ and $N\dv$ on each side. In this concrete instance, one transferred arc from the left interface and one from the right interface land fully on $S$. The resulting secondary structure on $S$ is the ``meaning'' of the sentence in the DNA semantics.
\end{example}

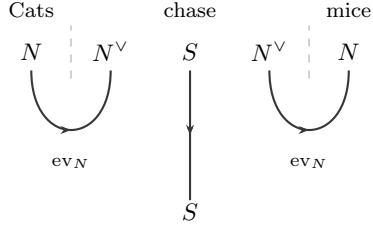
\begin{center}
\begin{tikzpicture}[x=0.7cm, y=0.7cm]

  \node[rigid node] at (1.5,0) {$N$};
  \node[rigid node] at (3,0) {$N^\vee$};
  \node[rigid node] at (4.5,0) {$S$};
  \node[rigid node] at (6,0) {$N^\vee$};
  \node[rigid node] at (7.5,0) {$N$};

  \node[rigid node] at (1.5,0.8) {\footnotesize Cats};
  \node[rigid node] at (4.5,0.8) {\footnotesize chase};
  \node[rigid node] at (7.5,0.8) {\footnotesize mice};

  \draw[dna wire plain, thick] (1.5,-0.35) .. controls +(0,-1.5) and +(0,-1.5) .. (3,-0.35);
  \node[rigid node] at (2.25,-2.1) {\scriptsize$\ev_N$};

  \draw[dna wire plain, thick] (6,-0.35) .. controls +(0,-1.5) and +(0,-1.5) .. (7.5,-0.35);
  \node[rigid node] at (6.75,-2.1) {\scriptsize$\ev_N$};

  \draw[dna wire, thick] (4.5,-0.35) -- (4.5,-2.8);
  \node[rigid node, below] at (4.5,-2.8) {$S$};

  \draw[dnaBorder, thin, dashed] (2.25,-0.5) -- (2.25,0.5);
  \draw[dnaBorder, thin, dashed] (6.75,-0.5) -- (6.75,0.5);
\end{tikzpicture}
\end{center}

\noindent
Concretely, let $N$ and $S$ be fixed DNA words of length 12. In this example, the last three positions of \emph{cats} and \emph{chase} are free, as are the first three positions of \emph{chase} and \emph{mice}; each lexical structure is non-trivial, maximally bound, and has no hairpins of size less than 3. The functor images of the lexical items are drawn as solid arcs; the dashed cups perform Watson--Crick evaluation on each $N$--$N\dv$ interface. The transfer result on $S$ contains at least one fully transferred arc from the right side, encoding the sentence meaning as a secondary structure on $S$.

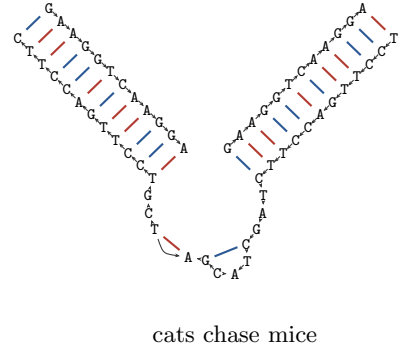
\begin{center}
\begin{minipage}{\textwidth}
  \centering
  \textbf{(a) Arc-evaluation diagram}
  \par\smallskip
  \resizebox{0.98\textwidth}{!}{
\begin{tikzpicture}[x=0.38cm, y=0.38cm]
  \node[base node, font=\scriptsize\ttfamily] (b0) at (1,0) {A};
  \node[base node, font=\scriptsize\ttfamily] (b1) at (2,0) {G};
  \node[base node, font=\scriptsize\ttfamily] (b2) at (3,0) {G};
  \node[base node, font=\scriptsize\ttfamily] (b3) at (4,0) {A};
  \node[base node, font=\scriptsize\ttfamily] (b4) at (5,0) {A};
  \node[base node, font=\scriptsize\ttfamily] (b5) at (6,0) {C};
  \node[base node, font=\scriptsize\ttfamily] (b6) at (7,0) {T};
  \node[base node, font=\scriptsize\ttfamily] (b7) at (8,0) {G};
  \node[base node, font=\scriptsize\ttfamily] (b8) at (9,0) {G};
  \node[base node, font=\scriptsize\ttfamily] (b9) at (10,0) {A};
  \node[base node, font=\scriptsize\ttfamily] (b10) at (11,0) {A};
  \node[base node, font=\scriptsize\ttfamily] (b11) at (12,0) {G};
  \node[base node, font=\scriptsize\ttfamily] (b12) at (13,0) {C};
  \node[base node, font=\scriptsize\ttfamily] (b13) at (14,0) {T};
  \node[base node, font=\scriptsize\ttfamily] (b14) at (15,0) {T};
  \node[base node, font=\scriptsize\ttfamily] (b15) at (16,0) {C};
  \node[base node, font=\scriptsize\ttfamily] (b16) at (17,0) {C};
  \node[base node, font=\scriptsize\ttfamily] (b17) at (18,0) {A};
  \node[base node, font=\scriptsize\ttfamily] (b18) at (19,0) {G};
  \node[base node, font=\scriptsize\ttfamily] (b19) at (20,0) {T};
  \node[base node, font=\scriptsize\ttfamily] (b20) at (21,0) {T};
  \node[base node, font=\scriptsize\ttfamily] (b21) at (22,0) {C};
  \node[base node, font=\scriptsize\ttfamily] (b22) at (23,0) {C};
  \node[base node, font=\scriptsize\ttfamily] (b23) at (24,0) {T};
  \node[base node, font=\scriptsize\ttfamily] (b24) at (25,0) {G};
  \node[base node, font=\scriptsize\ttfamily] (b25) at (26,0) {C};
  \node[base node, font=\scriptsize\ttfamily] (b26) at (27,0) {T};
  \node[base node, font=\scriptsize\ttfamily] (b27) at (28,0) {A};
  \node[base node, font=\scriptsize\ttfamily] (b28) at (29,0) {G};
  \node[base node, font=\scriptsize\ttfamily] (b29) at (30,0) {C};
  \node[base node, font=\scriptsize\ttfamily] (b30) at (31,0) {A};
  \node[base node, font=\scriptsize\ttfamily] (b31) at (32,0) {T};
  \node[base node, font=\scriptsize\ttfamily] (b32) at (33,0) {C};
  \node[base node, font=\scriptsize\ttfamily] (b33) at (34,0) {G};
  \node[base node, font=\scriptsize\ttfamily] (b34) at (35,0) {A};
  \node[base node, font=\scriptsize\ttfamily] (b35) at (36,0) {T};
  \node[base node, font=\scriptsize\ttfamily] (b36) at (37,0) {C};
  \node[base node, font=\scriptsize\ttfamily] (b37) at (38,0) {T};
  \node[base node, font=\scriptsize\ttfamily] (b38) at (39,0) {T};
  \node[base node, font=\scriptsize\ttfamily] (b39) at (40,0) {C};
  \node[base node, font=\scriptsize\ttfamily] (b40) at (41,0) {C};
  \node[base node, font=\scriptsize\ttfamily] (b41) at (42,0) {A};
  \node[base node, font=\scriptsize\ttfamily] (b42) at (43,0) {G};
  \node[base node, font=\scriptsize\ttfamily] (b43) at (44,0) {T};
  \node[base node, font=\scriptsize\ttfamily] (b44) at (45,0) {T};
  \node[base node, font=\scriptsize\ttfamily] (b45) at (46,0) {C};
  \node[base node, font=\scriptsize\ttfamily] (b46) at (47,0) {C};
  \node[base node, font=\scriptsize\ttfamily] (b47) at (48,0) {T};
  \node[base node, font=\scriptsize\ttfamily] (b48) at (49,0) {A};
  \node[base node, font=\scriptsize\ttfamily] (b49) at (50,0) {G};
  \node[base node, font=\scriptsize\ttfamily] (b50) at (51,0) {G};
  \node[base node, font=\scriptsize\ttfamily] (b51) at (52,0) {A};
  \node[base node, font=\scriptsize\ttfamily] (b52) at (53,0) {A};
  \node[base node, font=\scriptsize\ttfamily] (b53) at (54,0) {C};
  \node[base node, font=\scriptsize\ttfamily] (b54) at (55,0) {T};
  \node[base node, font=\scriptsize\ttfamily] (b55) at (56,0) {G};
  \node[base node, font=\scriptsize\ttfamily] (b56) at (57,0) {G};
  \node[base node, font=\scriptsize\ttfamily] (b57) at (58,0) {A};
  \node[base node, font=\scriptsize\ttfamily] (b58) at (59,0) {A};
  \node[base node, font=\scriptsize\ttfamily] (b59) at (60,0) {G};
  \node[rigid node] at (6.5,-0.7) {\normalsize$N$};
  \node[rigid node] at (18.5,-0.7) {\normalsize$N^\vee$};
  \node[rigid node] at (30.5,-0.7) {\normalsize$S$};
  \node[rigid node] at (42.5,-0.7) {\normalsize$N^\vee$};
  \node[rigid node] at (54.5,-0.7) {\normalsize$N$};
  \node[rigid node] at (6.5,1.0) {\normalsize Cats};
  \node[rigid node] at (30.5,1.0) {\normalsize chase};
  \node[rigid node] at (54.5,1.0) {\normalsize mice};
  \draw[dnaBorder, thin, dashed] (12.5,-0.5) -- (12.5,0.5);
  \draw[dnaBorder, thin, dashed] (48.5,-0.5) -- (48.5,0.5);
  \draw[at arc rev] (b0.north) .. controls +(0,1.20) and +(0,1.20) .. (b6.north);
  \draw[cg arc] (b1.north) .. controls +(0,0.55) and +(0,0.55) .. (b5.north);
  \draw[cg arc rev] (b16.north) .. controls +(0,1.85) and +(0,1.85) .. (b24.north);
  \draw[at arc rev] (b17.north) .. controls +(0,1.20) and +(0,1.20) .. (b23.north);
  \draw[cg arc] (b18.north) .. controls +(0,0.55) and +(0,0.55) .. (b22.north);
  \draw[cg arc rev] (b25.north) .. controls +(0,2.50) and +(0,2.50) .. (b42.north);
  \draw[at arc] (b26.north) .. controls +(0,1.85) and +(0,1.85) .. (b41.north);
  \draw[at arc rev] (b27.north) .. controls +(0,1.20) and +(0,1.20) .. (b37.north);
  \draw[cg arc] (b28.north) .. controls +(0,0.55) and +(0,0.55) .. (b32.north);
  \draw[cg arc rev] (b53.north) .. controls +(0,1.20) and +(0,1.20) .. (b59.north);
  \draw[at arc] (b54.north) .. controls +(0,0.55) and +(0,0.55) .. (b58.north);
  \draw[at arc, densely dashed] (b0.south) .. controls +(0,-7.70) and +(0,-7.70) .. (b23.south);
  \draw[cg arc rev, densely dashed] (b1.south) .. controls +(0,-7.05) and +(0,-7.05) .. (b22.south);
  \draw[cg arc rev, densely dashed] (b2.south) .. controls +(0,-6.40) and +(0,-6.40) .. (b21.south);
  \draw[at arc, densely dashed] (b3.south) .. controls +(0,-5.75) and +(0,-5.75) .. (b20.south);
  \draw[at arc, densely dashed] (b4.south) .. controls +(0,-5.10) and +(0,-5.10) .. (b19.south);
  \draw[cg arc, densely dashed] (b5.south) .. controls +(0,-4.45) and +(0,-4.45) .. (b18.south);
  \draw[at arc rev, densely dashed] (b6.south) .. controls +(0,-3.80) and +(0,-3.80) .. (b17.south);
  \draw[cg arc rev, densely dashed] (b7.south) .. controls +(0,-3.15) and +(0,-3.15) .. (b16.south);
  \draw[cg arc rev, densely dashed] (b8.south) .. controls +(0,-2.50) and +(0,-2.50) .. (b15.south);
  \draw[at arc, densely dashed] (b9.south) .. controls +(0,-1.85) and +(0,-1.85) .. (b14.south);
  \draw[at arc, densely dashed] (b10.south) .. controls +(0,-1.20) and +(0,-1.20) .. (b13.south);
  \draw[cg arc rev, densely dashed] (b11.south) .. controls +(0,-0.55) and +(0,-0.55) .. (b12.south);
  \draw[cg arc, densely dashed] (b36.south) .. controls +(0,-7.70) and +(0,-7.70) .. (b59.south);
  \draw[at arc rev, densely dashed] (b37.south) .. controls +(0,-7.05) and +(0,-7.05) .. (b58.south);
  \draw[at arc rev, densely dashed] (b38.south) .. controls +(0,-6.40) and +(0,-6.40) .. (b57.south);
  \draw[cg arc, densely dashed] (b39.south) .. controls +(0,-5.75) and +(0,-5.75) .. (b56.south);
  \draw[cg arc, densely dashed] (b40.south) .. controls +(0,-5.10) and +(0,-5.10) .. (b55.south);
  \draw[at arc, densely dashed] (b41.south) .. controls +(0,-4.45) and +(0,-4.45) .. (b54.south);
  \draw[cg arc rev, densely dashed] (b42.south) .. controls +(0,-3.80) and +(0,-3.80) .. (b53.south);
  \draw[at arc rev, densely dashed] (b43.south) .. controls +(0,-3.15) and +(0,-3.15) .. (b52.south);
  \draw[at arc rev, densely dashed] (b44.south) .. controls +(0,-2.50) and +(0,-2.50) .. (b51.south);
  \draw[cg arc, densely dashed] (b45.south) .. controls +(0,-1.85) and +(0,-1.85) .. (b50.south);
  \draw[cg arc, densely dashed] (b46.south) .. controls +(0,-1.20) and +(0,-1.20) .. (b49.south);
  \draw[at arc rev, densely dashed] (b47.south) .. controls +(0,-0.55) and +(0,-0.55) .. (b48.south);
  \node[rigid node] at (30.5,-2.0) {\normalsize$\id_S$};
  \node[rigid node] at (12.5,-9.2) {\normalsize$\ev_N$};
  \node[rigid node] at (48.5,-9.2) {\normalsize$\ev_N$};
  \node[rigid eq] at (30.5,-10.2) {$\longrightarrow$};
  \node[base node, font=\scriptsize\ttfamily] (r0) at (25.0,-12.7) {G};
  \node[base node, font=\scriptsize\ttfamily] (r1) at (26.0,-12.7) {C};
  \node[base node, font=\scriptsize\ttfamily] (r2) at (27.0,-12.7) {T};
  \node[base node, font=\scriptsize\ttfamily] (r3) at (28.0,-12.7) {A};
  \node[base node, font=\scriptsize\ttfamily] (r4) at (29.0,-12.7) {G};
  \node[base node, font=\scriptsize\ttfamily] (r5) at (30.0,-12.7) {C};
  \node[base node, font=\scriptsize\ttfamily] (r6) at (31.0,-12.7) {A};
  \node[base node, font=\scriptsize\ttfamily] (r7) at (32.0,-12.7) {T};
  \node[base node, font=\scriptsize\ttfamily] (r8) at (33.0,-12.7) {C};
  \node[base node, font=\scriptsize\ttfamily] (r9) at (34.0,-12.7) {G};
  \node[base node, font=\scriptsize\ttfamily] (r10) at (35.0,-12.7) {A};
  \node[base node, font=\scriptsize\ttfamily] (r11) at (36.0,-12.7) {T};
  \node[rigid node] at (30.5,-13.4) {\normalsize$S$};
  \draw[at arc] (r2.north) .. controls +(0,0.55) and +(0,0.55) .. (r3.north);
  \draw[cg arc] (r4.north) .. controls +(0,0.55) and +(0,0.55) .. (r8.north);
\end{tikzpicture}}
\end{minipage}

\vspace{0.9em}

\begin{minipage}{0.49\textwidth}
  \centering
  \textbf{(b) Folded full sentence structure}
  \par\smallskip
\begin{tikzpicture}[x=0.36cm, y=0.36cm]
  \node[base node, font=\tiny\ttfamily] (b0) at (1.474,0.749) {A};
  \node[base node, font=\tiny\ttfamily] (b1) at (0.825,0.707) {G};
  \node[base node, font=\tiny\ttfamily] (b2) at (0.224,0.461) {G};
  \node[base node, font=\tiny\ttfamily] (b3) at (0.000,-0.149) {A};
  \node[base node, font=\tiny\ttfamily] (b4) at (0.301,-0.725) {A};
  \node[base node, font=\tiny\ttfamily] (b5) at (0.930,-0.890) {C};
  \node[base node, font=\tiny\ttfamily] (b6) at (1.579,-0.847) {T};
  \node[base node, font=\tiny\ttfamily] (b7) at (2.181,-1.092) {G};
  \node[base node, font=\tiny\ttfamily] (b8) at (2.816,-0.953) {G};
  \node[base node, font=\tiny\ttfamily] (b9) at (3.260,-0.478) {A};
  \node[base node, font=\tiny\ttfamily] (b10) at (3.358,0.164) {A};
  \node[base node, font=\tiny\ttfamily] (b11) at (3.074,0.749) {G};
  \draw[-{Stealth[length=1.6pt,width=1.2pt]}, dnaWire, thin] (b0) -- (b1);
  \draw[-{Stealth[length=1.6pt,width=1.2pt]}, dnaWire, thin] (b1) -- (b2);
  \draw[-{Stealth[length=1.6pt,width=1.2pt]}, dnaWire, thin] (b2) -- (b3);
  \draw[-{Stealth[length=1.6pt,width=1.2pt]}, dnaWire, thin] (b3) -- (b4);
  \draw[-{Stealth[length=1.6pt,width=1.2pt]}, dnaWire, thin] (b4) -- (b5);
  \draw[-{Stealth[length=1.6pt,width=1.2pt]}, dnaWire, thin] (b5) -- (b6);
  \draw[-{Stealth[length=1.6pt,width=1.2pt]}, dnaWire, thin] (b6) -- (b7);
  \draw[-{Stealth[length=1.6pt,width=1.2pt]}, dnaWire, thin] (b7) -- (b8);
  \draw[-{Stealth[length=1.6pt,width=1.2pt]}, dnaWire, thin] (b8) -- (b9);
  \draw[-{Stealth[length=1.6pt,width=1.2pt]}, dnaWire, thin] (b9) -- (b10);
  \draw[-{Stealth[length=1.6pt,width=1.2pt]}, dnaWire, thin] (b10) -- (b11);
  \draw[dnaAT, thick] (b0) -- (b6);
  \draw[dnaCG, thick] (b1) -- (b5);
  \node[base node, font=\tiny\ttfamily] (b12) at (7.147,1.647) {C};
  \node[base node, font=\tiny\ttfamily] (b13) at (6.622,1.263) {T};
  \node[base node, font=\tiny\ttfamily] (b14) at (6.264,0.721) {T};
  \node[base node, font=\tiny\ttfamily] (b15) at (6.118,0.087) {C};
  \node[base node, font=\tiny\ttfamily] (b16) at (6.203,-0.557) {C};
  \node[base node, font=\tiny\ttfamily] (b17) at (5.733,-1.006) {A};
  \node[base node, font=\tiny\ttfamily] (b18) at (5.262,-1.454) {G};
  \node[base node, font=\tiny\ttfamily] (b19) at (4.958,-2.029) {T};
  \node[base node, font=\tiny\ttfamily] (b20) at (5.178,-2.640) {T};
  \node[base node, font=\tiny\ttfamily] (b21) at (5.778,-2.889) {C};
  \node[base node, font=\tiny\ttfamily] (b22) at (6.366,-2.612) {C};
  \node[base node, font=\tiny\ttfamily] (b23) at (6.837,-2.164) {T};
  \node[base node, font=\tiny\ttfamily] (b24) at (7.307,-1.715) {G};
  \node[base node, font=\tiny\ttfamily] (b25) at (7.947,-1.831) {C};
  \node[base node, font=\tiny\ttfamily] (b26) at (8.231,-2.415) {T};
  \node[base node, font=\tiny\ttfamily] (b27) at (8.232,-3.065) {A};
  \node[base node, font=\tiny\ttfamily] (b28) at (7.683,-3.414) {G};
  \node[base node, font=\tiny\ttfamily] (b29) at (7.068,-3.626) {C};
  \node[base node, font=\tiny\ttfamily] (b30) at (6.812,-4.223) {A};
  \node[base node, font=\tiny\ttfamily] (b31) at (7.082,-4.815) {T};
  \node[base node, font=\tiny\ttfamily] (b32) at (7.700,-5.013) {C};
  \node[base node, font=\tiny\ttfamily] (b33) at (8.257,-5.350) {G};
  \node[base node, font=\tiny\ttfamily] (b34) at (8.907,-5.342) {A};
  \node[base node, font=\tiny\ttfamily] (b35) at (9.456,-4.994) {T};
  \node[base node, font=\tiny\ttfamily] (b36) at (9.739,-4.409) {C};
  \node[base node, font=\tiny\ttfamily] (b37) at (9.672,-3.763) {T};
  \node[base node, font=\tiny\ttfamily] (b38) at (10.182,-3.360) {T};
  \node[base node, font=\tiny\ttfamily] (b39) at (10.374,-2.739) {C};
  \node[base node, font=\tiny\ttfamily] (b40) at (10.181,-2.118) {C};
  \node[base node, font=\tiny\ttfamily] (b41) at (9.670,-1.716) {A};
  \node[base node, font=\tiny\ttfamily] (b42) at (9.386,-1.131) {G};
  \node[base node, font=\tiny\ttfamily] (b43) at (9.691,-0.557) {T};
  \node[base node, font=\tiny\ttfamily] (b44) at (9.776,0.087) {T};
  \node[base node, font=\tiny\ttfamily] (b45) at (9.630,0.721) {C};
  \node[base node, font=\tiny\ttfamily] (b46) at (9.272,1.263) {C};
  \node[base node, font=\tiny\ttfamily] (b47) at (8.747,1.647) {T};
  \draw[-{Stealth[length=1.6pt,width=1.2pt]}, dnaWire, thin] (b12) -- (b13);
  \draw[-{Stealth[length=1.6pt,width=1.2pt]}, dnaWire, thin] (b13) -- (b14);
  \draw[-{Stealth[length=1.6pt,width=1.2pt]}, dnaWire, thin] (b14) -- (b15);
  \draw[-{Stealth[length=1.6pt,width=1.2pt]}, dnaWire, thin] (b15) -- (b16);
  \draw[-{Stealth[length=1.6pt,width=1.2pt]}, dnaWire, thin] (b16) -- (b17);
  \draw[-{Stealth[length=1.6pt,width=1.2pt]}, dnaWire, thin] (b17) -- (b18);
  \draw[-{Stealth[length=1.6pt,width=1.2pt]}, dnaWire, thin] (b18) -- (b19);
  \draw[-{Stealth[length=1.6pt,width=1.2pt]}, dnaWire, thin] (b19) -- (b20);
  \draw[-{Stealth[length=1.6pt,width=1.2pt]}, dnaWire, thin] (b20) -- (b21);
  \draw[-{Stealth[length=1.6pt,width=1.2pt]}, dnaWire, thin] (b21) -- (b22);
  \draw[-{Stealth[length=1.6pt,width=1.2pt]}, dnaWire, thin] (b22) -- (b23);
  \draw[-{Stealth[length=1.6pt,width=1.2pt]}, dnaWire, thin] (b23) -- (b24);
  \draw[-{Stealth[length=1.6pt,width=1.2pt]}, dnaWire, thin] (b24) -- (b25);
  \draw[-{Stealth[length=1.6pt,width=1.2pt]}, dnaWire, thin] (b25) -- (b26);
  \draw[-{Stealth[length=1.6pt,width=1.2pt]}, dnaWire, thin] (b26) -- (b27);
  \draw[-{Stealth[length=1.6pt,width=1.2pt]}, dnaWire, thin] (b27) -- (b28);
  \draw[-{Stealth[length=1.6pt,width=1.2pt]}, dnaWire, thin] (b28) -- (b29);
  \draw[-{Stealth[length=1.6pt,width=1.2pt]}, dnaWire, thin] (b29) -- (b30);
  \draw[-{Stealth[length=1.6pt,width=1.2pt]}, dnaWire, thin] (b30) -- (b31);
  \draw[-{Stealth[length=1.6pt,width=1.2pt]}, dnaWire, thin] (b31) -- (b32);
  \draw[-{Stealth[length=1.6pt,width=1.2pt]}, dnaWire, thin] (b32) -- (b33);
  \draw[-{Stealth[length=1.6pt,width=1.2pt]}, dnaWire, thin] (b33) -- (b34);
  \draw[-{Stealth[length=1.6pt,width=1.2pt]}, dnaWire, thin] (b34) -- (b35);
  \draw[-{Stealth[length=1.6pt,width=1.2pt]}, dnaWire, thin] (b35) -- (b36);
  \draw[-{Stealth[length=1.6pt,width=1.2pt]}, dnaWire, thin] (b36) -- (b37);
  \draw[-{Stealth[length=1.6pt,width=1.2pt]}, dnaWire, thin] (b37) -- (b38);
  \draw[-{Stealth[length=1.6pt,width=1.2pt]}, dnaWire, thin] (b38) -- (b39);
  \draw[-{Stealth[length=1.6pt,width=1.2pt]}, dnaWire, thin] (b39) -- (b40);
  \draw[-{Stealth[length=1.6pt,width=1.2pt]}, dnaWire, thin] (b40) -- (b41);
  \draw[-{Stealth[length=1.6pt,width=1.2pt]}, dnaWire, thin] (b41) -- (b42);
  \draw[-{Stealth[length=1.6pt,width=1.2pt]}, dnaWire, thin] (b42) -- (b43);
  \draw[-{Stealth[length=1.6pt,width=1.2pt]}, dnaWire, thin] (b43) -- (b44);
  \draw[-{Stealth[length=1.6pt,width=1.2pt]}, dnaWire, thin] (b44) -- (b45);
  \draw[-{Stealth[length=1.6pt,width=1.2pt]}, dnaWire, thin] (b45) -- (b46);
  \draw[-{Stealth[length=1.6pt,width=1.2pt]}, dnaWire, thin] (b46) -- (b47);
  \draw[dnaCG, thick] (b16) -- (b24);
  \draw[dnaAT, thick] (b17) -- (b23);
  \draw[dnaCG, thick] (b18) -- (b22);
  \draw[dnaCG, thick] (b25) -- (b42);
  \draw[dnaAT, thick] (b26) -- (b41);
  \draw[dnaAT, thick] (b27) -- (b37);
  \draw[dnaCG, thick] (b28) -- (b32);
  \node[base node, font=\tiny\ttfamily] (b48) at (12.258,0.749) {A};
  \node[base node, font=\tiny\ttfamily] (b49) at (11.974,0.164) {G};
  \node[base node, font=\tiny\ttfamily] (b50) at (12.072,-0.478) {G};
  \node[base node, font=\tiny\ttfamily] (b51) at (12.516,-0.953) {A};
  \node[base node, font=\tiny\ttfamily] (b52) at (13.151,-1.092) {A};
  \node[base node, font=\tiny\ttfamily] (b53) at (13.753,-0.847) {C};
  \node[base node, font=\tiny\ttfamily] (b54) at (14.402,-0.890) {T};
  \node[base node, font=\tiny\ttfamily] (b55) at (15.031,-0.725) {G};
  \node[base node, font=\tiny\ttfamily] (b56) at (15.332,-0.149) {G};
  \node[base node, font=\tiny\ttfamily] (b57) at (15.108,0.461) {A};
  \node[base node, font=\tiny\ttfamily] (b58) at (14.507,0.707) {A};
  \node[base node, font=\tiny\ttfamily] (b59) at (13.858,0.749) {G};
  \draw[-{Stealth[length=1.6pt,width=1.2pt]}, dnaWire, thin] (b48) -- (b49);
  \draw[-{Stealth[length=1.6pt,width=1.2pt]}, dnaWire, thin] (b49) -- (b50);
  \draw[-{Stealth[length=1.6pt,width=1.2pt]}, dnaWire, thin] (b50) -- (b51);
  \draw[-{Stealth[length=1.6pt,width=1.2pt]}, dnaWire, thin] (b51) -- (b52);
  \draw[-{Stealth[length=1.6pt,width=1.2pt]}, dnaWire, thin] (b52) -- (b53);
  \draw[-{Stealth[length=1.6pt,width=1.2pt]}, dnaWire, thin] (b53) -- (b54);
  \draw[-{Stealth[length=1.6pt,width=1.2pt]}, dnaWire, thin] (b54) -- (b55);
  \draw[-{Stealth[length=1.6pt,width=1.2pt]}, dnaWire, thin] (b55) -- (b56);
  \draw[-{Stealth[length=1.6pt,width=1.2pt]}, dnaWire, thin] (b56) -- (b57);
  \draw[-{Stealth[length=1.6pt,width=1.2pt]}, dnaWire, thin] (b57) -- (b58);
  \draw[-{Stealth[length=1.6pt,width=1.2pt]}, dnaWire, thin] (b58) -- (b59);
  \draw[dnaCG, thick] (b53) -- (b59);
  \draw[dnaAT, thick] (b54) -- (b58);
  \node[rigid node, font=\small] at ($(b6)+(0.0,-1.8)$) {cats};
  \node[rigid node, font=\small] at ($(b30)+(0.0,-2.0)$) {chase};
  \node[rigid node, font=\small] at ($(b54)+(0.0,-1.8)$) {mice};
\end{tikzpicture}
\end{minipage}\hfill
\begin{minipage}{0.49\textwidth}
  \centering
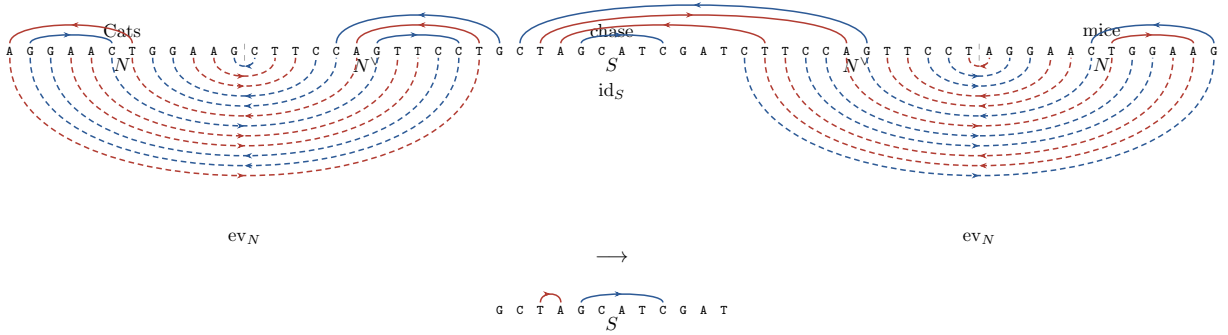
  \textbf{(c) Full folded transfer state on the sentence sequence}
  \par\smallskip
\begin{tikzpicture}[x=0.36cm, y=0.36cm]
  \node[base node, font=\tiny\ttfamily] (b0) at (-0.800,1.980) {A};
  \node[base node, font=\tiny\ttfamily] (b1) at (-1.251,2.448) {G};
  \node[base node, font=\tiny\ttfamily] (b2) at (-1.703,2.916) {G};
  \node[base node, font=\tiny\ttfamily] (b3) at (-2.154,3.383) {A};
  \node[base node, font=\tiny\ttfamily] (b4) at (-2.606,3.851) {A};
  \node[base node, font=\tiny\ttfamily] (b5) at (-3.057,4.319) {C};
  \node[base node, font=\tiny\ttfamily] (b6) at (-3.509,4.786) {T};
  \node[base node, font=\tiny\ttfamily] (b7) at (-3.960,5.254) {G};
  \node[base node, font=\tiny\ttfamily] (b8) at (-4.412,5.721) {G};
  \node[base node, font=\tiny\ttfamily] (b9) at (-4.863,6.189) {A};
  \node[base node, font=\tiny\ttfamily] (b10) at (-5.315,6.657) {A};
  \node[base node, font=\tiny\ttfamily] (b11) at (-5.766,7.124) {G};
  \node[base node, font=\tiny\ttfamily] (b12) at (-6.917,6.013) {C};
  \node[base node, font=\tiny\ttfamily] (b13) at (-6.466,5.545) {T};
  \node[base node, font=\tiny\ttfamily] (b14) at (-6.014,5.078) {T};
  \node[base node, font=\tiny\ttfamily] (b15) at (-5.563,4.610) {C};
  \node[base node, font=\tiny\ttfamily] (b16) at (-5.111,4.142) {C};
  \node[base node, font=\tiny\ttfamily] (b17) at (-4.660,3.675) {A};
  \node[base node, font=\tiny\ttfamily] (b18) at (-4.208,3.207) {G};
  \node[base node, font=\tiny\ttfamily] (b19) at (-3.757,2.740) {T};
  \node[base node, font=\tiny\ttfamily] (b20) at (-3.305,2.272) {T};
  \node[base node, font=\tiny\ttfamily] (b21) at (-2.854,1.804) {C};
  \node[base node, font=\tiny\ttfamily] (b22) at (-2.403,1.337) {C};
  \node[base node, font=\tiny\ttfamily] (b23) at (-1.951,0.869) {T};
  \node[base node, font=\tiny\ttfamily] (b24) at (-2.122,0.242) {G};
  \node[base node, font=\tiny\ttfamily] (b25) at (-2.097,-0.407) {C};
  \node[base node, font=\tiny\ttfamily] (b26) at (-1.877,-1.019) {T};
  \node[base node, font=\tiny\ttfamily] (b27) at (-0.642,-2.037) {A};
  \node[base node, font=\tiny\ttfamily] (b28) at (0.000,-2.136) {G};
  \node[base node, font=\tiny\ttfamily] (b29) at (0.422,-2.631) {C};
  \node[base node, font=\tiny\ttfamily] (b30) at (1.071,-2.652) {A};
  \node[base node, font=\tiny\ttfamily] (b31) at (1.524,-2.185) {T};
  \node[base node, font=\tiny\ttfamily] (b32) at (1.484,-1.537) {C};
  \node[base node, font=\tiny\ttfamily] (b33) at (1.877,-1.019) {G};
  \node[base node, font=\tiny\ttfamily] (b34) at (2.097,-0.407) {A};
  \node[base node, font=\tiny\ttfamily] (b35) at (2.122,0.242) {T};
  \node[base node, font=\tiny\ttfamily] (b36) at (1.951,0.869) {C};
  \node[base node, font=\tiny\ttfamily] (b37) at (2.403,1.337) {T};
  \node[base node, font=\tiny\ttfamily] (b38) at (2.854,1.804) {T};
  \node[base node, font=\tiny\ttfamily] (b39) at (3.305,2.272) {C};
  \node[base node, font=\tiny\ttfamily] (b40) at (3.757,2.740) {C};
  \node[base node, font=\tiny\ttfamily] (b41) at (4.208,3.207) {A};
  \node[base node, font=\tiny\ttfamily] (b42) at (4.660,3.675) {G};
  \node[base node, font=\tiny\ttfamily] (b43) at (5.111,4.142) {T};
  \node[base node, font=\tiny\ttfamily] (b44) at (5.563,4.610) {T};
  \node[base node, font=\tiny\ttfamily] (b45) at (6.014,5.078) {C};
  \node[base node, font=\tiny\ttfamily] (b46) at (6.466,5.545) {C};
  \node[base node, font=\tiny\ttfamily] (b47) at (6.917,6.013) {T};
  \node[base node, font=\tiny\ttfamily] (b48) at (5.766,7.124) {A};
  \node[base node, font=\tiny\ttfamily] (b49) at (5.315,6.657) {G};
  \node[base node, font=\tiny\ttfamily] (b50) at (4.863,6.189) {G};
  \node[base node, font=\tiny\ttfamily] (b51) at (4.412,5.721) {A};
  \node[base node, font=\tiny\ttfamily] (b52) at (3.960,5.254) {A};
  \node[base node, font=\tiny\ttfamily] (b53) at (3.509,4.786) {C};
  \node[base node, font=\tiny\ttfamily] (b54) at (3.057,4.319) {T};
  \node[base node, font=\tiny\ttfamily] (b55) at (2.606,3.851) {G};
  \node[base node, font=\tiny\ttfamily] (b56) at (2.154,3.383) {G};
  \node[base node, font=\tiny\ttfamily] (b57) at (1.703,2.916) {A};
  \node[base node, font=\tiny\ttfamily] (b58) at (1.251,2.448) {A};
  \node[base node, font=\tiny\ttfamily] (b59) at (0.800,1.980) {G};
  \draw[-{Stealth[length=1.6pt,width=1.2pt]}, dnaWire, thin] (b0) -- (b1);
  \draw[-{Stealth[length=1.6pt,width=1.2pt]}, dnaWire, thin] (b1) -- (b2);
  \draw[-{Stealth[length=1.6pt,width=1.2pt]}, dnaWire, thin] (b2) -- (b3);
  \draw[-{Stealth[length=1.6pt,width=1.2pt]}, dnaWire, thin] (b3) -- (b4);
  \draw[-{Stealth[length=1.6pt,width=1.2pt]}, dnaWire, thin] (b4) -- (b5);
  \draw[-{Stealth[length=1.6pt,width=1.2pt]}, dnaWire, thin] (b5) -- (b6);
  \draw[-{Stealth[length=1.6pt,width=1.2pt]}, dnaWire, thin] (b6) -- (b7);
  \draw[-{Stealth[length=1.6pt,width=1.2pt]}, dnaWire, thin] (b7) -- (b8);
  \draw[-{Stealth[length=1.6pt,width=1.2pt]}, dnaWire, thin] (b8) -- (b9);
  \draw[-{Stealth[length=1.6pt,width=1.2pt]}, dnaWire, thin] (b9) -- (b10);
  \draw[-{Stealth[length=1.6pt,width=1.2pt]}, dnaWire, thin] (b10) -- (b11);
  \draw[-{Stealth[length=1.6pt,width=1.2pt]}, dnaWire, thin] (b12) -- (b13);
  \draw[-{Stealth[length=1.6pt,width=1.2pt]}, dnaWire, thin] (b13) -- (b14);
  \draw[-{Stealth[length=1.6pt,width=1.2pt]}, dnaWire, thin] (b14) -- (b15);
  \draw[-{Stealth[length=1.6pt,width=1.2pt]}, dnaWire, thin] (b15) -- (b16);
  \draw[-{Stealth[length=1.6pt,width=1.2pt]}, dnaWire, thin] (b16) -- (b17);
  \draw[-{Stealth[length=1.6pt,width=1.2pt]}, dnaWire, thin] (b17) -- (b18);
  \draw[-{Stealth[length=1.6pt,width=1.2pt]}, dnaWire, thin] (b18) -- (b19);
  \draw[-{Stealth[length=1.6pt,width=1.2pt]}, dnaWire, thin] (b19) -- (b20);
  \draw[-{Stealth[length=1.6pt,width=1.2pt]}, dnaWire, thin] (b20) -- (b21);
  \draw[-{Stealth[length=1.6pt,width=1.2pt]}, dnaWire, thin] (b21) -- (b22);
  \draw[-{Stealth[length=1.6pt,width=1.2pt]}, dnaWire, thin] (b22) -- (b23);
  \draw[-{Stealth[length=1.6pt,width=1.2pt]}, dnaWire, thin] (b23) -- (b24);
  \draw[-{Stealth[length=1.6pt,width=1.2pt]}, dnaWire, thin] (b24) -- (b25);
  \draw[-{Stealth[length=1.6pt,width=1.2pt]}, dnaWire, thin] (b25) -- (b26);
  \draw[-{Stealth[length=1.6pt,width=1.2pt]}, dnaWire, thin] (b26) .. controls (-1.570,-1.904) and (-1.570,-1.904) .. (b27);
  \draw[-{Stealth[length=1.6pt,width=1.2pt]}, dnaWire, thin] (b27) -- (b28);
  \draw[-{Stealth[length=1.6pt,width=1.2pt]}, dnaWire, thin] (b28) -- (b29);
  \draw[-{Stealth[length=1.6pt,width=1.2pt]}, dnaWire, thin] (b29) -- (b30);
  \draw[-{Stealth[length=1.6pt,width=1.2pt]}, dnaWire, thin] (b30) -- (b31);
  \draw[-{Stealth[length=1.6pt,width=1.2pt]}, dnaWire, thin] (b31) -- (b32);
  \draw[-{Stealth[length=1.6pt,width=1.2pt]}, dnaWire, thin] (b32) -- (b33);
  \draw[-{Stealth[length=1.6pt,width=1.2pt]}, dnaWire, thin] (b33) -- (b34);
  \draw[-{Stealth[length=1.6pt,width=1.2pt]}, dnaWire, thin] (b34) -- (b35);
  \draw[-{Stealth[length=1.6pt,width=1.2pt]}, dnaWire, thin] (b35) -- (b36);
  \draw[-{Stealth[length=1.6pt,width=1.2pt]}, dnaWire, thin] (b36) -- (b37);
  \draw[-{Stealth[length=1.6pt,width=1.2pt]}, dnaWire, thin] (b37) -- (b38);
  \draw[-{Stealth[length=1.6pt,width=1.2pt]}, dnaWire, thin] (b38) -- (b39);
  \draw[-{Stealth[length=1.6pt,width=1.2pt]}, dnaWire, thin] (b39) -- (b40);
  \draw[-{Stealth[length=1.6pt,width=1.2pt]}, dnaWire, thin] (b40) -- (b41);
  \draw[-{Stealth[length=1.6pt,width=1.2pt]}, dnaWire, thin] (b41) -- (b42);
  \draw[-{Stealth[length=1.6pt,width=1.2pt]}, dnaWire, thin] (b42) -- (b43);
  \draw[-{Stealth[length=1.6pt,width=1.2pt]}, dnaWire, thin] (b43) -- (b44);
  \draw[-{Stealth[length=1.6pt,width=1.2pt]}, dnaWire, thin] (b44) -- (b45);
  \draw[-{Stealth[length=1.6pt,width=1.2pt]}, dnaWire, thin] (b45) -- (b46);
  \draw[-{Stealth[length=1.6pt,width=1.2pt]}, dnaWire, thin] (b46) -- (b47);
  \draw[-{Stealth[length=1.6pt,width=1.2pt]}, dnaWire, thin] (b48) -- (b49);
  \draw[-{Stealth[length=1.6pt,width=1.2pt]}, dnaWire, thin] (b49) -- (b50);
  \draw[-{Stealth[length=1.6pt,width=1.2pt]}, dnaWire, thin] (b50) -- (b51);
  \draw[-{Stealth[length=1.6pt,width=1.2pt]}, dnaWire, thin] (b51) -- (b52);
  \draw[-{Stealth[length=1.6pt,width=1.2pt]}, dnaWire, thin] (b52) -- (b53);
  \draw[-{Stealth[length=1.6pt,width=1.2pt]}, dnaWire, thin] (b53) -- (b54);
  \draw[-{Stealth[length=1.6pt,width=1.2pt]}, dnaWire, thin] (b54) -- (b55);
  \draw[-{Stealth[length=1.6pt,width=1.2pt]}, dnaWire, thin] (b55) -- (b56);
  \draw[-{Stealth[length=1.6pt,width=1.2pt]}, dnaWire, thin] (b56) -- (b57);
  \draw[-{Stealth[length=1.6pt,width=1.2pt]}, dnaWire, thin] (b57) -- (b58);
  \draw[-{Stealth[length=1.6pt,width=1.2pt]}, dnaWire, thin] (b58) -- (b59);
  \draw[dnaAT, thick] (b0) -- (b23);
  \draw[dnaCG, thick] (b1) -- (b22);
  \draw[dnaCG, thick] (b2) -- (b21);
  \draw[dnaAT, thick] (b3) -- (b20);
  \draw[dnaAT, thick] (b4) -- (b19);
  \draw[dnaCG, thick] (b5) -- (b18);
  \draw[dnaAT, thick] (b6) -- (b17);
  \draw[dnaCG, thick] (b7) -- (b16);
  \draw[dnaCG, thick] (b8) -- (b15);
  \draw[dnaAT, thick] (b9) -- (b14);
  \draw[dnaAT, thick] (b10) -- (b13);
  \draw[dnaCG, thick] (b11) -- (b12);
  \draw[dnaAT, thick] (b26) -- (b27);
  \draw[dnaCG, thick] (b28) -- (b32);
  \draw[dnaCG, thick] (b36) -- (b59);
  \draw[dnaAT, thick] (b37) -- (b58);
  \draw[dnaAT, thick] (b38) -- (b57);
  \draw[dnaCG, thick] (b39) -- (b56);
  \draw[dnaCG, thick] (b40) -- (b55);
  \draw[dnaAT, thick] (b41) -- (b54);
  \draw[dnaCG, thick] (b42) -- (b53);
  \draw[dnaAT, thick] (b43) -- (b52);
  \draw[dnaAT, thick] (b44) -- (b51);
  \draw[dnaCG, thick] (b45) -- (b50);
  \draw[dnaCG, thick] (b46) -- (b49);
  \draw[dnaAT, thick] (b47) -- (b48);
  \node[rigid node, font=\small] at ($(b30)+(0.0,-2.2)$) {cats chase mice};
\end{tikzpicture}
\end{minipage}
\end{center}

\begin{remark}
In the straightened view, grammatical reduction---the contraction of adjacent type-antitype pairs---is precisely the zip-and-transfer operation of Section~\ref{subsec:strand} applied to complementary DNA interfaces. Parsing a sentence thus corresponds to a cascade of zip-and-transfer steps, structurally analogous to the strand-displacement cascades used in DNA computing~\cite{Qian2011, Seelig2006}.
\end{remark}

\section{Discussion}\label{sec:remarks}

Several natural variants of $\Ddna$ suggest themselves. Rather than erasing closed loops, one can assign scalar weights $\delta_{AT}$, $\delta_{CG}$ to produce a category enriched over a scalar ring, enabling thermodynamic refinements via nearest-neighbor energy models~\cite{SantaLucia1998}. Dropping the noncrossing constraint would model pseudoknotted structures; categorically, this amounts to adding a braiding, moving from a pivotal to a braided or tortile monoidal category~\cite{Kassel1995}. One could also replace literal sequences by equivalence classes (e.g.\ families satisfying a design relation), so that a type image $F(n)$ denotes a domain rather than a single fixed word.

The short sequences used in the grammar examples of Section~\ref{sec:grammar} are purely illustrative. In practice, physically stable type assignments would require sequences tens of nucleotides long, with meaning stored in stems exceeding the thermal stability threshold~\cite{SantaLucia1998}. Longer complementary interfaces open a richer design space: one can program selective binding profiles by tuning sequences, as exploited by DNA molecular computation~\cite{Seelig2006, Zhang2009, Qian2011, Soloveichik2010} and systematic sequence design tools such as NUPACK~\cite{Zadeh2011}.

The categorical framework complements several existing lines of work. Winfree~\cite{Winfree1998} connects the physical structure of DNA assemblies to the Chomsky hierarchy, showing that tree-like self-assembly generates exactly the context-free languages; our noncrossing diagrams inhabit the same structural regime, though a full formal equivalence remains to be established. Phillips and Cardelli~\cite{Phillips2009} describe a programming language for composable strand-displacement circuits under an explicit no-secondary-structure assumption for strands; $\Ddna$ removes that restriction, so it is strictly more general while retaining an algebraic account of the same compositional operations, with zip-and-transfer as the primitive. Fontana, Schuster, and collaborators~\cite{Fontana1993a, Schuster1994} studied RNA secondary structures as phenotypes under a genotype--phenotype map induced by maximum-bond (equivalently, energy-minimizing) folding; in our setting, that map can itself be read as a computation semantics on structures. The algorithmic-chemistry program and its constructive dynamical extension~\cite{Fontana1992, FontanaBuss1994, FontanaBuss1996} then supply the complementary viewpoint: interacting typed entities compute by composition while constructing new entities. $\Ddna$ realizes this point through duality: terms are secondary structures and, at the same time, morphisms (functions), so zip-and-transfer is simultaneously chemical interaction and functional composition.

We close with three open questions. (1)~Which functors $\mathcal{G} \to \Ddna$ produce linguistically useful semantic distinctions? (2)~For what classes of problems or computations is this model naturally well suited? (3)~What are the thermodynamic limitations and sequence-design constraints of this framework, and can weighted-loop conventions and/or enriched morphisms encode nearest-neighbor free energies through local loop changes under composition while extending to a braided pivotal setting that accommodates pseudoknots?


\section*{Acknowledgements}
This work was first inspired during the author's doctoral studies in the laboratory of Erik Winfree at Caltech and conceived while the author was an Omidyar Postdoctoral Fellow at the Santa Fe Institute.
The author gratefully acknowledges the financial support of the Fontana Lab at the Department of Systems Biology, Harvard Medical School, during the development of this work, and thanks Walter Fontana personally for his encouragement and intellectual generosity.
The author also thanks Martha Lewis and Emily Riehl for useful conversations during the early stages of this project.


\bibliographystyle{abbrvurl}
\bibliography{refs}

@article{Coecke2010,
  author    = {Bob Coecke and Mehrnoosh Sadrzadeh and Stephen Clark},
  title     = {Mathematical Foundations for a Compositional Distributional Model of Meaning},
  journal   = {Linguistic Analysis},
  volume    = {36},
  pages     = {345--384},
  year      = {2010},
  note      = {\href{https://arxiv.org/abs/1003.4394}{arXiv:1003.4394}},

}

@inproceedings{Fontana1992,
  author    = {Walter Fontana},
  title     = {Algorithmic Chemistry: A Model for Functional Self-Organization},
  booktitle = {Artificial Life {II}},
  editor    = {Christopher G. Langton and Charles Taylor and J. Doyne Farmer and Steen Rasmussen},
  publisher = {Addison-Wesley},
  year      = {1992},

}

@article{Fontana1993a,
  author    = {Walter Fontana and Peter F. Stadler and Elena G. Bornberg-Bauer and Thomas Griesmacher and Ivo L. Hofacker and Manfred Tacker and Pedro Tarazona and Erik D. Weinberger and Peter Schuster},
  title     = {{RNA} Folding and Combinatory Landscapes},
  journal   = {Physical Review E},
  volume    = {47},
  pages     = {2083--2099},
  year      = {1993},
  doi      = {10.1103/PhysRevE.47.2083},

}

@article{FontanaBuss1994,
  author    = {Walter Fontana and Leo W. Buss},
  title     = {The Arrival of the Fittest: Toward a Theory of Biological Organization},
  journal   = {Bulletin of Mathematical Biology},
  volume    = {56},
  pages     = {1--64},
  year      = {1994},
  doi      = {10.1016/S0092-8240(05)80205-8},

}

@incollection{FontanaBuss1996,
  author    = {Walter Fontana and Leo W. Buss},
  title     = {The Barrier of Objects: From Dynamical Systems to Bounded Organizations},
  booktitle = {Boundaries and Barriers},
  editor    = {John Casti and Anders Karlqvist},
  pages     = {56--116},
  publisher = {Addison-Wesley},
  year      = {1996},

}

@article{Joyal1991,
  author    = {Andr{\'e} Joyal and Ross Street},
  title     = {The Geometry of Tensor Calculus, {I}},
  journal   = {Advances in Mathematics},
  volume    = {88},
  pages     = {55--112},
  year      = {1991},
  doi      = {10.1016/0001-8708(91)90003-P},

}

@book{Kassel1995,
  author    = {Christian Kassel},
  title     = {Quantum Groups},
  series    = {Graduate Texts in Mathematics},
  volume    = {155},
  publisher = {Springer},
  year      = {1995},
  doi      = {10.1007/978-1-4612-0783-2},

}

@inproceedings{Lambek1999,
  author    = {Joachim Lambek},
  title     = {Type Grammar Revisited},
  booktitle = {Logical Aspects of Computational Linguistics},
  series    = {Lecture Notes in Computer Science},
  volume    = {1582},
  pages     = {1--27},
  publisher = {Springer},
  year      = {1999},
  doi      = {10.1007/3-540-48975-4_1},

}

@article{Phillips2009,
  author    = {Andrew Phillips and Luca Cardelli},
  title     = {A Programming Language for Composable {DNA} Circuits},
  journal   = {Journal of the Royal Society Interface},
  volume    = {6},
  pages     = {S419--S436},
  year      = {2009},
  doi      = {10.1098/rsif.2009.0072.focus},

}

@article{Qian2011,
  author    = {Lulu Qian and Erik Winfree},
  title     = {Scaling Up Digital Circuit Computation with {DNA} Strand Displacement Cascades},
  journal   = {Science},
  volume    = {332},
  pages     = {1196--1201},
  year      = {2011},
  doi      = {10.1126/science.1200520},

}

@article{SantaLucia1998,
  author    = {John {SantaLucia, Jr.}},
  title     = {A Unified View of Polymer, Dumbbell, and Oligonucleotide {DNA} Nearest-Neighbor Thermodynamics},
  journal   = {Proceedings of the National Academy of Sciences},
  volume    = {95},
  pages     = {1460--1465},
  year      = {1998},
  doi      = {10.1073/pnas.95.4.1460},

}

@article{Schuster1994,
  author    = {Peter Schuster and Walter Fontana and Peter F. Stadler and Ivo L. Hofacker},
  title     = {From Sequences to Shapes and Back: A Case Study in {RNA} Secondary Structures},
  journal   = {Proceedings of the Royal Society B},
  volume    = {255},
  pages     = {279--284},
  year      = {1994},
  doi      = {10.1098/rspb.1994.0040},

}

@article{Seelig2006,
  author    = {Georg Seelig and David Soloveichik and David Yu Zhang and Erik Winfree},
  title     = {Enzyme-Free Nucleic Acid Logic Circuits},
  journal   = {Science},
  volume    = {314},
  pages     = {1585--1588},
  year      = {2006},
  doi      = {10.1126/science.1132493},

}

@incollection{Selinger2011,
  author    = {Peter Selinger},
  title     = {A Survey of Graphical Languages for Monoidal Categories},
  booktitle = {New Structures for Physics},
  series    = {Lecture Notes in Physics},
  volume    = {813},
  pages     = {289--355},
  publisher = {Springer},
  year      = {2011},
  doi      = {10.1007/978-3-642-12821-9_4},

}

@article{Soloveichik2010,
  author    = {David Soloveichik and Georg Seelig and Erik Winfree},
  title     = {{DNA} as a Universal Substrate for Chemical Kinetics},
  journal   = {Proceedings of the National Academy of Sciences},
  volume    = {107},
  pages     = {5393--5398},
  year      = {2010},
  doi      = {10.1073/pnas.0909380107},

}

@article{Waterman1978,
  author    = {Michael S. Waterman},
  title     = {Secondary Structure of Single-Stranded Nucleic Acids},
  journal   = {Advances in Mathematics, Supplementary Studies},
  volume    = {1},
  pages     = {167--212},
  year      = {1978},

}

@phdthesis{Winfree1998,
  author    = {Erik Winfree},
  title     = {Algorithmic Self-Assembly of {DNA}},
  school    = {California Institute of Technology},
  year      = {1998},

}

@article{Yurke2000,
  author    = {Bernard Yurke and Andrew J. Turberfield and Allen P. {Mills, Jr.} and Friedrich C. Simmel and Jennifer L. Neumann},
  title     = {A {DNA}-Fuelled Molecular Machine Made of {DNA}},
  journal   = {Nature},
  volume    = {406},
  pages     = {605--608},
  year      = {2000},
  doi      = {10.1038/35017555},

}

@article{Zadeh2011,
  author    = {Joseph N. Zadeh and Conrad D. Steenberg and Justin S. Bois and Brian R. Wolfe and Marshall B. Pierce and Asif R. Khan and Robert M. Dirks and Niles A. Pierce},
  title     = {{NUPACK}: Analysis and Design of Nucleic Acid Systems},
  journal   = {Journal of Computational Chemistry},
  volume    = {32},
  pages     = {170--173},
  year      = {2011},
  doi      = {10.1002/jcc.21596},

}

@article{Zhang2009,
  author    = {David Yu Zhang and Erik Winfree},
  title     = {Control of {DNA} Strand Displacement Kinetics Using Toehold Exchange},
  journal   = {Journal of the American Chemical Society},
  volume    = {131},
  pages     = {17303--17314},
  year      = {2009},
  doi      = {10.1021/ja906987s},

}

@article{Radding1982,
  author    = {Charles M. Radding},
  title     = {Homologous Pairing and Strand Exchange in Genetic Recombination},
  journal   = {Annual Review of Genetics},
  volume    = {16},
  pages     = {405--437},
  year      = {1982},
  doi      = {10.1146/annurev.ge.16.120182.002201},

}


\appendix
\section{Worked toehold-mediated strand displacement}\label{app:tmsd}

\begin{example}[Toehold-mediated strand displacement]\label{ex:tmsd}
Let $t$ be a toehold and $s$ a branch migration domain. Consider
\[
f \colon s \to ts,
\qquad
g \colon ts \to \eps,
\]
where $f$ has through-strands on the $s$ positions (the target toehold positions are initially unmatched) and $g$ is the empty morphism. As a concrete instance, set
\[
t=\texttt{GAA},\qquad s=\texttt{TCTCTC},\qquad ts=\texttt{GAATCTCTC}.
\]
Then $\widehat{f} \colon \eps \to s\dv \otimes ts$ forms the substrate duplex on the $s$ segment, while $\widehat{g} \colon \eps \to (ts)\dv = s\dv t\dv$ is a free invader strand. In folded form, the initial state has three strands
\[
(ts)\dv,\qquad ts,\qquad s\dv,
\]
with $(ts)\dv$ free and $ts$ paired with $s\dv$ via $\coev_{s\dv}$. Contracting $ts \otimes (ts)\dv$ performs toehold binding first and then branch migration along $s$. At the folded-complex level, this transfer yields the duplex $(ts)\dv \otimes ts$ together with released $s\dv$.
\end{example}

\begin{center}
\begin{minipage}{\textwidth}
  \centering
  \textbf{(a) Arc-evaluation diagram}
  \par\smallskip
  \resizebox{0.68\textwidth}{!}{
\begin{tikzpicture}[x=0.72cm, y=0.72cm]
  \begin{scope}[on background layer]
    \fill[dnaInterface] (15.5,-3.0) rectangle (15.5,2.0);
  \end{scope}
  \node[base node] (b1) at (1,0) {G};
  \node[base node] (b2) at (2,0) {A};
  \node[base node] (b3) at (3,0) {G};
  \node[base node] (b4) at (4,0) {A};
  \node[base node] (b5) at (5,0) {G};
  \node[base node] (b6) at (6,0) {A};
  \node[base node] (b7) at (7,0) {G};
  \node[base node] (b8) at (8,0) {A};
  \node[base node] (b9) at (9,0) {A};
  \node[base node] (b10) at (10,0) {T};
  \node[base node] (b11) at (11,0) {C};
  \node[base node] (b12) at (12,0) {T};
  \node[base node] (b13) at (13,0) {C};
  \node[base node] (b14) at (14,0) {T};
  \node[base node] (b15) at (15,0) {C};
  \node[base node] (b16) at (16,0) {G};
  \node[base node] (b17) at (17,0) {A};
  \node[base node] (b18) at (18,0) {G};
  \node[base node] (b19) at (19,0) {A};
  \node[base node] (b20) at (20,0) {G};
  \node[base node] (b21) at (21,0) {A};
  \node[base node] (b22) at (22,0) {T};
  \node[base node] (b23) at (23,0) {T};
  \node[base node] (b24) at (24,0) {C};
  \node[rigid node] at (20.0,-0.65) {\footnotesize$y^\vee$};
  \node[rigid node] at (11.0,-0.65) {\footnotesize$y$};
  \node[rigid node] at (3.5,-0.65) {\footnotesize$x^\vee$};
  \draw[cg arc] (b1.north) .. controls +(0,3.80) and +(0,3.80) .. (b15.north);
  \draw[at arc rev] (b2.north) .. controls +(0,3.15) and +(0,3.15) .. (b14.north);
  \draw[cg arc] (b3.north) .. controls +(0,2.50) and +(0,2.50) .. (b13.north);
  \draw[at arc rev] (b4.north) .. controls +(0,1.85) and +(0,1.85) .. (b12.north);
  \draw[cg arc] (b5.north) .. controls +(0,1.20) and +(0,1.20) .. (b11.north);
  \draw[at arc rev] (b6.north) .. controls +(0,0.55) and +(0,0.55) .. (b10.north);
  \draw[cg arc, densely dashed] (b7.south) .. controls +(0,-5.75) and +(0,-5.75) .. (b24.south);
  \draw[at arc rev, densely dashed] (b8.south) .. controls +(0,-5.10) and +(0,-5.10) .. (b23.south);
  \draw[at arc rev, densely dashed] (b9.south) .. controls +(0,-4.45) and +(0,-4.45) .. (b22.south);
  \draw[at arc, densely dashed] (b10.south) .. controls +(0,-3.80) and +(0,-3.80) .. (b21.south);
  \draw[cg arc rev, densely dashed] (b11.south) .. controls +(0,-3.15) and +(0,-3.15) .. (b20.south);
  \draw[at arc, densely dashed] (b12.south) .. controls +(0,-2.50) and +(0,-2.50) .. (b19.south);
  \draw[cg arc rev, densely dashed] (b13.south) .. controls +(0,-1.85) and +(0,-1.85) .. (b18.south);
  \draw[at arc, densely dashed] (b14.south) .. controls +(0,-1.20) and +(0,-1.20) .. (b17.south);
  \draw[cg arc rev, densely dashed] (b15.south) .. controls +(0,-0.55) and +(0,-0.55) .. (b16.south);
  \node[rigid node] at (8.0,2.2) {\footnotesize$\widehat{f}$};
  \node[rigid node] at (20.0,2.2) {\footnotesize$\widehat{g}$};
  \node[rigid node, rotate=90] at (15.5,1.2) {\footnotesize interface};
  \node[rigid eq] at (12.5,-3.6) {$\longrightarrow$};
  \node[base node] (r1) at (10.0,-5.2) {G};
  \node[base node] (r2) at (11.0,-5.2) {A};
  \node[base node] (r3) at (12.0,-5.2) {G};
  \node[base node] (r4) at (13.0,-5.2) {A};
  \node[base node] (r5) at (14.0,-5.2) {G};
  \node[base node] (r6) at (15.0,-5.2) {A};
  \node[rigid node] at (12.5,-5.85) {\footnotesize$x^\vee$};
\end{tikzpicture}}
\end{minipage}

\vspace{0.9em}

\begin{minipage}{0.32\textwidth}
  \centering
  \textbf{(b) Folded initial state: free $(ts)^\vee$ and duplex $ts|s^\vee$}
  \par\smallskip
  \scalebox{1.0}{
\begin{tikzpicture}[x=0.5cm, y=0.5cm]
  \node[base node, font=\tiny\ttfamily] (b0) at (0.000,0.000) {G};
  \node[base node, font=\tiny\ttfamily] (b1) at (0.650,0.000) {A};
  \node[base node, font=\tiny\ttfamily] (b2) at (1.300,0.000) {G};
  \node[base node, font=\tiny\ttfamily] (b3) at (1.950,0.000) {A};
  \node[base node, font=\tiny\ttfamily] (b4) at (2.600,0.000) {G};
  \node[base node, font=\tiny\ttfamily] (b5) at (3.250,0.000) {A};
  \node[base node, font=\tiny\ttfamily] (b6) at (3.900,0.000) {T};
  \node[base node, font=\tiny\ttfamily] (b7) at (4.550,0.000) {T};
  \node[base node, font=\tiny\ttfamily] (b8) at (5.200,0.000) {C};
  \draw[-{Stealth[length=1.6pt,width=1.2pt]}, dnaWire, thin] (b0) -- (b1);
  \draw[-{Stealth[length=1.6pt,width=1.2pt]}, dnaWire, thin] (b1) -- (b2);
  \draw[-{Stealth[length=1.6pt,width=1.2pt]}, dnaWire, thin] (b2) -- (b3);
  \draw[-{Stealth[length=1.6pt,width=1.2pt]}, dnaWire, thin] (b3) -- (b4);
  \draw[-{Stealth[length=1.6pt,width=1.2pt]}, dnaWire, thin] (b4) -- (b5);
  \draw[-{Stealth[length=1.6pt,width=1.2pt]}, dnaWire, thin] (b5) -- (b6);
  \draw[-{Stealth[length=1.6pt,width=1.2pt]}, dnaWire, thin] (b6) -- (b7);
  \draw[-{Stealth[length=1.6pt,width=1.2pt]}, dnaWire, thin] (b7) -- (b8);
  \node[base node, font=\tiny\ttfamily] (b9) at (6.915,0.475) {G};
  \node[base node, font=\tiny\ttfamily] (b10) at (6.800,-0.165) {A};
  \node[base node, font=\tiny\ttfamily] (b11) at (7.131,-0.724) {A};
  \node[base node, font=\tiny\ttfamily] (b12) at (7.748,-0.930) {T};
  \node[base node, font=\tiny\ttfamily] (b13) at (8.318,-1.241) {C};
  \node[base node, font=\tiny\ttfamily] (b14) at (8.889,-1.553) {T};
  \node[base node, font=\tiny\ttfamily] (b15) at (9.459,-1.864) {C};
  \node[base node, font=\tiny\ttfamily] (b16) at (10.030,-2.176) {T};
  \node[base node, font=\tiny\ttfamily] (b17) at (10.600,-2.488) {C};
  \node[base node, font=\tiny\ttfamily] (b18) at (11.367,-1.084) {G};
  \node[base node, font=\tiny\ttfamily] (b19) at (10.797,-0.772) {A};
  \node[base node, font=\tiny\ttfamily] (b20) at (10.226,-0.460) {G};
  \node[base node, font=\tiny\ttfamily] (b21) at (9.656,-0.149) {A};
  \node[base node, font=\tiny\ttfamily] (b22) at (9.085,0.163) {G};
  \node[base node, font=\tiny\ttfamily] (b23) at (8.515,0.475) {A};
  \draw[-{Stealth[length=1.6pt,width=1.2pt]}, dnaWire, thin] (b9) -- (b10);
  \draw[-{Stealth[length=1.6pt,width=1.2pt]}, dnaWire, thin] (b10) -- (b11);
  \draw[-{Stealth[length=1.6pt,width=1.2pt]}, dnaWire, thin] (b11) -- (b12);
  \draw[-{Stealth[length=1.6pt,width=1.2pt]}, dnaWire, thin] (b12) -- (b13);
  \draw[-{Stealth[length=1.6pt,width=1.2pt]}, dnaWire, thin] (b13) -- (b14);
  \draw[-{Stealth[length=1.6pt,width=1.2pt]}, dnaWire, thin] (b14) -- (b15);
  \draw[-{Stealth[length=1.6pt,width=1.2pt]}, dnaWire, thin] (b15) -- (b16);
  \draw[-{Stealth[length=1.6pt,width=1.2pt]}, dnaWire, thin] (b16) -- (b17);
  \draw[-{Stealth[length=1.6pt,width=1.2pt]}, dnaWire, thin] (b17) .. controls (11.411,-2.019) and (11.411,-2.019) .. (b18);
  \draw[-{Stealth[length=1.6pt,width=1.2pt]}, dnaWire, thin] (b18) -- (b19);
  \draw[-{Stealth[length=1.6pt,width=1.2pt]}, dnaWire, thin] (b19) -- (b20);
  \draw[-{Stealth[length=1.6pt,width=1.2pt]}, dnaWire, thin] (b20) -- (b21);
  \draw[-{Stealth[length=1.6pt,width=1.2pt]}, dnaWire, thin] (b21) -- (b22);
  \draw[-{Stealth[length=1.6pt,width=1.2pt]}, dnaWire, thin] (b22) -- (b23);
  \draw[dnaAT, thick] (b12) -- (b23);
  \draw[dnaCG, thick] (b13) -- (b22);
  \draw[dnaAT, thick] (b14) -- (b21);
  \draw[dnaCG, thick] (b15) -- (b20);
  \draw[dnaAT, thick] (b16) -- (b19);
  \draw[dnaCG, thick] (b17) -- (b18);
\end{tikzpicture}}
\end{minipage}\hfill
\begin{minipage}{0.32\textwidth}
  \centering
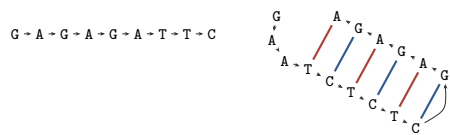
  \textbf{(c) Midway toehold binding and branch migration}
  \par\smallskip
  \scalebox{1.0}{
\begin{tikzpicture}[x=0.5cm, y=0.5cm]
  \node[base node, font=\tiny\ttfamily] (b0) at (-0.800,1.308) {G};
  \node[base node, font=\tiny\ttfamily] (b1) at (-1.270,0.859) {A};
  \node[base node, font=\tiny\ttfamily] (b2) at (-1.512,0.256) {G};
  \node[base node, font=\tiny\ttfamily] (b3) at (-1.482,-0.393) {A};
  \node[base node, font=\tiny\ttfamily] (b4) at (-1.186,-0.972) {G};
  \node[base node, font=\tiny\ttfamily] (b5) at (-1.400,-1.586) {A};
  \node[base node, font=\tiny\ttfamily] (b6) at (-1.614,-2.200) {T};
  \node[base node, font=\tiny\ttfamily] (b7) at (-1.828,-2.813) {T};
  \node[base node, font=\tiny\ttfamily] (b8) at (-2.041,-3.427) {C};
  \node[base node, font=\tiny\ttfamily] (b9) at (-0.531,-3.954) {G};
  \node[base node, font=\tiny\ttfamily] (b10) at (-0.317,-3.340) {A};
  \node[base node, font=\tiny\ttfamily] (b11) at (-0.103,-2.726) {A};
  \node[base node, font=\tiny\ttfamily] (b12) at (0.111,-2.112) {T};
  \node[base node, font=\tiny\ttfamily] (b13) at (0.325,-1.499) {C};
  \node[base node, font=\tiny\ttfamily] (b14) at (0.917,-1.229) {T};
  \node[base node, font=\tiny\ttfamily] (b15) at (1.520,-1.471) {C};
  \node[base node, font=\tiny\ttfamily] (b16) at (2.123,-1.713) {T};
  \node[base node, font=\tiny\ttfamily] (b17) at (2.727,-1.955) {C};
  \node[base node, font=\tiny\ttfamily] (b18) at (3.322,-0.470) {G};
  \node[base node, font=\tiny\ttfamily] (b19) at (2.719,-0.228) {A};
  \node[base node, font=\tiny\ttfamily] (b20) at (2.115,0.014) {G};
  \node[base node, font=\tiny\ttfamily] (b21) at (1.512,0.256) {A};
  \node[base node, font=\tiny\ttfamily] (b22) at (1.270,0.859) {G};
  \node[base node, font=\tiny\ttfamily] (b23) at (0.800,1.308) {A};
  \draw[-{Stealth[length=1.6pt,width=1.2pt]}, dnaWire, thin] (b0) -- (b1);
  \draw[-{Stealth[length=1.6pt,width=1.2pt]}, dnaWire, thin] (b1) -- (b2);
  \draw[-{Stealth[length=1.6pt,width=1.2pt]}, dnaWire, thin] (b2) -- (b3);
  \draw[-{Stealth[length=1.6pt,width=1.2pt]}, dnaWire, thin] (b3) -- (b4);
  \draw[-{Stealth[length=1.6pt,width=1.2pt]}, dnaWire, thin] (b4) -- (b5);
  \draw[-{Stealth[length=1.6pt,width=1.2pt]}, dnaWire, thin] (b5) -- (b6);
  \draw[-{Stealth[length=1.6pt,width=1.2pt]}, dnaWire, thin] (b6) -- (b7);
  \draw[-{Stealth[length=1.6pt,width=1.2pt]}, dnaWire, thin] (b7) -- (b8);
  \draw[-{Stealth[length=1.6pt,width=1.2pt]}, dnaWire, thin] (b9) -- (b10);
  \draw[-{Stealth[length=1.6pt,width=1.2pt]}, dnaWire, thin] (b10) -- (b11);
  \draw[-{Stealth[length=1.6pt,width=1.2pt]}, dnaWire, thin] (b11) -- (b12);
  \draw[-{Stealth[length=1.6pt,width=1.2pt]}, dnaWire, thin] (b12) -- (b13);
  \draw[-{Stealth[length=1.6pt,width=1.2pt]}, dnaWire, thin] (b13) -- (b14);
  \draw[-{Stealth[length=1.6pt,width=1.2pt]}, dnaWire, thin] (b14) -- (b15);
  \draw[-{Stealth[length=1.6pt,width=1.2pt]}, dnaWire, thin] (b15) -- (b16);
  \draw[-{Stealth[length=1.6pt,width=1.2pt]}, dnaWire, thin] (b16) -- (b17);
  \draw[-{Stealth[length=1.6pt,width=1.2pt]}, dnaWire, thin] (b18) -- (b19);
  \draw[-{Stealth[length=1.6pt,width=1.2pt]}, dnaWire, thin] (b19) -- (b20);
  \draw[-{Stealth[length=1.6pt,width=1.2pt]}, dnaWire, thin] (b20) -- (b21);
  \draw[-{Stealth[length=1.6pt,width=1.2pt]}, dnaWire, thin] (b21) -- (b22);
  \draw[-{Stealth[length=1.6pt,width=1.2pt]}, dnaWire, thin] (b22) -- (b23);
  \draw[dnaCG, thick] (b4) -- (b13);
  \draw[dnaAT, thick] (b5) -- (b12);
  \draw[dnaAT, thick] (b6) -- (b11);
  \draw[dnaAT, thick] (b7) -- (b10);
  \draw[dnaCG, thick] (b8) -- (b9);
  \draw[dnaAT, thick] (b14) -- (b21);
  \draw[dnaCG, thick] (b15) -- (b20);
  \draw[dnaAT, thick] (b16) -- (b19);
  \draw[dnaCG, thick] (b17) -- (b18);
\end{tikzpicture}}
\end{minipage}\hfill
\begin{minipage}{0.32\textwidth}
  \centering
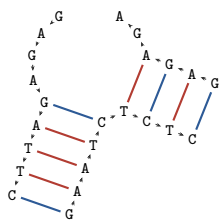
  \textbf{(d) Folded transfer result: duplex $(ts)^\vee|ts$ and free $s^\vee$}
  \par\smallskip
  \scalebox{1.0}{
\begin{tikzpicture}[x=0.5cm, y=0.5cm]
  \node[base node, font=\tiny\ttfamily] (b0) at (0.000,-0.000) {G};
  \node[base node, font=\tiny\ttfamily] (b1) at (0.000,-0.650) {A};
  \node[base node, font=\tiny\ttfamily] (b2) at (0.000,-1.300) {G};
  \node[base node, font=\tiny\ttfamily] (b3) at (0.000,-1.950) {A};
  \node[base node, font=\tiny\ttfamily] (b4) at (0.000,-2.600) {G};
  \node[base node, font=\tiny\ttfamily] (b5) at (0.000,-3.250) {A};
  \node[base node, font=\tiny\ttfamily] (b6) at (0.000,-3.900) {T};
  \node[base node, font=\tiny\ttfamily] (b7) at (0.000,-4.550) {T};
  \node[base node, font=\tiny\ttfamily] (b8) at (0.000,-5.200) {C};
  \node[base node, font=\tiny\ttfamily] (b9) at (1.600,-5.200) {G};
  \node[base node, font=\tiny\ttfamily] (b10) at (1.600,-4.550) {A};
  \node[base node, font=\tiny\ttfamily] (b11) at (1.600,-3.900) {A};
  \node[base node, font=\tiny\ttfamily] (b12) at (1.600,-3.250) {T};
  \node[base node, font=\tiny\ttfamily] (b13) at (1.600,-2.600) {C};
  \node[base node, font=\tiny\ttfamily] (b14) at (1.600,-1.950) {T};
  \node[base node, font=\tiny\ttfamily] (b15) at (1.600,-1.300) {C};
  \node[base node, font=\tiny\ttfamily] (b16) at (1.600,-0.650) {T};
  \node[base node, font=\tiny\ttfamily] (b17) at (1.600,0.000) {C};
  \draw[-{Stealth[length=1.6pt,width=1.2pt]}, dnaWire, thin] (b0) -- (b1);
  \draw[-{Stealth[length=1.6pt,width=1.2pt]}, dnaWire, thin] (b1) -- (b2);
  \draw[-{Stealth[length=1.6pt,width=1.2pt]}, dnaWire, thin] (b2) -- (b3);
  \draw[-{Stealth[length=1.6pt,width=1.2pt]}, dnaWire, thin] (b3) -- (b4);
  \draw[-{Stealth[length=1.6pt,width=1.2pt]}, dnaWire, thin] (b4) -- (b5);
  \draw[-{Stealth[length=1.6pt,width=1.2pt]}, dnaWire, thin] (b5) -- (b6);
  \draw[-{Stealth[length=1.6pt,width=1.2pt]}, dnaWire, thin] (b6) -- (b7);
  \draw[-{Stealth[length=1.6pt,width=1.2pt]}, dnaWire, thin] (b7) -- (b8);
  \draw[-{Stealth[length=1.6pt,width=1.2pt]}, dnaWire, thin] (b8) .. controls (0.800,-5.688) and (0.800,-5.688) .. (b9);
  \draw[-{Stealth[length=1.6pt,width=1.2pt]}, dnaWire, thin] (b9) -- (b10);
  \draw[-{Stealth[length=1.6pt,width=1.2pt]}, dnaWire, thin] (b10) -- (b11);
  \draw[-{Stealth[length=1.6pt,width=1.2pt]}, dnaWire, thin] (b11) -- (b12);
  \draw[-{Stealth[length=1.6pt,width=1.2pt]}, dnaWire, thin] (b12) -- (b13);
  \draw[-{Stealth[length=1.6pt,width=1.2pt]}, dnaWire, thin] (b13) -- (b14);
  \draw[-{Stealth[length=1.6pt,width=1.2pt]}, dnaWire, thin] (b14) -- (b15);
  \draw[-{Stealth[length=1.6pt,width=1.2pt]}, dnaWire, thin] (b15) -- (b16);
  \draw[-{Stealth[length=1.6pt,width=1.2pt]}, dnaWire, thin] (b16) -- (b17);
  \draw[dnaCG, thick] (b0) -- (b17);
  \draw[dnaAT, thick] (b1) -- (b16);
  \draw[dnaCG, thick] (b2) -- (b15);
  \draw[dnaAT, thick] (b3) -- (b14);
  \draw[dnaCG, thick] (b4) -- (b13);
  \draw[dnaAT, thick] (b5) -- (b12);
  \draw[dnaAT, thick] (b6) -- (b11);
  \draw[dnaAT, thick] (b7) -- (b10);
  \draw[dnaCG, thick] (b8) -- (b9);
  \node[base node, font=\tiny\ttfamily] (b18) at (3.200,0.000) {G};
  \node[base node, font=\tiny\ttfamily] (b19) at (3.850,0.000) {A};
  \node[base node, font=\tiny\ttfamily] (b20) at (4.500,0.000) {G};
  \node[base node, font=\tiny\ttfamily] (b21) at (5.150,0.000) {A};
  \node[base node, font=\tiny\ttfamily] (b22) at (5.800,0.000) {G};
  \node[base node, font=\tiny\ttfamily] (b23) at (6.450,0.000) {A};
  \draw[-{Stealth[length=1.6pt,width=1.2pt]}, dnaWire, thin] (b18) -- (b19);
  \draw[-{Stealth[length=1.6pt,width=1.2pt]}, dnaWire, thin] (b19) -- (b20);
  \draw[-{Stealth[length=1.6pt,width=1.2pt]}, dnaWire, thin] (b20) -- (b21);
  \draw[-{Stealth[length=1.6pt,width=1.2pt]}, dnaWire, thin] (b21) -- (b22);
  \draw[-{Stealth[length=1.6pt,width=1.2pt]}, dnaWire, thin] (b22) -- (b23);
\end{tikzpicture}}
\end{minipage}
\end{center}

\end{document}